\documentclass[11pt,journal]{IEEEtran}
\usepackage{amsmath,graphicx,bm,amssymb,amsthm,enumerate,cases}
\renewcommand{\(}{\left(}
\renewcommand{\)}{\right)}
\renewcommand{\[}{\left[}
\renewcommand{\]}{\right]}
\renewcommand{\c}{\mathbf{c}}

\renewcommand{\b}{\mathbf{b}}

\renewcommand{\j}{\mathbf{j}}

\newcommand{\s}{\mathbf{s}}

\renewcommand{\k}{\mathbf{k}}

\newcommand{\W}{\mathbf{W}}

\newcommand{\D}{\mathbf{D}}

\newcommand{\x}{\mathbf{x}}

\renewcommand{\t}{\mathbf{t}}
\newcommand{\A}{\mathbf{A}}
\newcommand{\T}{\mathbf{T}}

\newcommand{\N}{\mathbb{N}}

\renewcommand{\u}{\mathbf{u}}
\renewcommand{\v}{\mathbf{v}}

\newcommand{\X}{\mathbf{X}}

\renewcommand{\b}{\mathbf{b}}

\renewcommand{\i}{\mathbf{i}}
\newcommand{\Tr}[1]{{\rm{Tr}}\left(#1\right)}

\newcommand{\End}[1]{{\rm{End}}}

\renewcommand{\log}[1]{{\rm{log}}#1}

\newcommand{\suppp}[1]{{\rm{supp}}\(#1\)}

\newcommand{\var}[1]{{\rm{var}}\(#1\)}

\newcommand{\mmod}{\,{\rm{mod}}\;}

\newtheorem{lemma}{Lemma}
\newtheorem{definition}{Definition}

\newtheorem{prop}{Proposition}
\newtheorem{corollary}{Corollary}

\newtheorem{rem}{Remark}

\newcommand{\norm}[1]{\left\lVert#1\right\rVert}

\usepackage{fontenc}
\usepackage{inputenc,cuted,bbm}
\usepackage[square,sort,compress,comma,numbers]{natbib}
\usepackage{epstopdf,pst-node}
\usepackage{flushend,stfloats}

\usepackage{enumitem}
\SetLabelAlign{Center}{{\makebox[1em]{\hfil#1\hfil}}}

\newcommand{\mypm}{\mathbin{\smash{%
\raisebox{0.35ex}{%
            $\underset{\raisebox{0.5ex}{$\smash -$}}{\smash+}$%
            }%
        }%
    }%
}

\begin{document}
\title{Symmetric Pseudo-Random Matrices}

\IEEEspecialpapernotice{\hfill\textit{Dedicated to the memory of Solomon W. Golomb (1932-2016)}}

\author{Ilya Soloveychik, Yu Xiang and Vahid Tarokh \\ John A. Paulson School of Engineering and Applied Sciences, \\ Harvard University
\thanks{This work was supported by the Fulbright Foundation and Army Research Office grant No. W911NF-15-1-0479.}
}
\maketitle

\begin{abstract}
We consider the problem of generating symmetric pseudo-random sign ($\mypm 1$) matrices based on the similarity of their spectra to Wigner's semicircular law. Using binary $m$-sequences (Golomb sequences) of lengths $n=2^m-1$, we give a simple explicit construction of circulant $n \times n$ sign matrices and show that their spectra converge to the semicircular law when $n$ grows. The Kolmogorov complexity of the proposed matrices equals to that of Golomb sequences and is at most $2\log_2(n)$ bits.
\end{abstract}

\begin{IEEEkeywords}
Pseudo-random matrices, semicircular law, Wigner ensemble.
\end{IEEEkeywords}

\section{Introduction}
\subsection{Wigner Matrices: Universality and Structure}
Random matrices have been a very active area of research for the last few decades and have found enormous applications in various areas of modern mathematics, physics, engineering, biological modeling, and other fields \cite{akemann2011oxford}. In this article, we focus on the square symmetric matrices with $\mypm 1$ entries, referred to as square symmetric {\it sign} matrices. For this model, Wigner \cite{wigner1955characteristic} demonstrated that if the elements of the upper triangular part (including the main diagonal) of an $n \times n$ matrix are independent Rademacher ($\mypm 1$ with equal probabilities) random variables, then as $n$ grows a properly scaled empirical spectral measure converges to the semi-circle law. Originally, Wigner proved convergence in expectation, but 3 years later he himself improved the result to convergence in probability \cite{wigner1958distribution}. In about a decade, Arnold \cite{arnold1967asymptotic} strengthened the claim to almost sure weak convergence. 

After the pioneering works of Wigner and the inception of the Random Matrix Theory (RMT), a great deal of research effort was put into generalization of the original basic random matrix setup. The two main directions of generalization can be roughly defined as \textit{universality} and \textit{structure}. 

The universality phenomenon can be viewed as a non-commutative analog of the Central Limit Theorem \cite{voiculescu1992free}, in other words, it can be understood as invariance of some limiting properties of the matrices under a change of the marginal (atomic) distributions of the matrix entries, given that the independence conditions are left intact. In the last 20 years, much work have been devoted to the analysis of the universality phenomena with regard to different matrix properties and of their exact bounds. Remarkably, when it comes to the semicircular limiting spectrum, necessary conditions on the marginal distributions were quite well studied (see \cite{soshnikov1999universality} and surveys \cite{erdHos2011universality, tao2012random}).

In the second direction of the generalization endeavor, which we call structure or dependence, scientists try to understand to what extent the tough independence conditions on the matrix entries can be relaxed without affecting the limiting behavior. The importance of this line of research is hard to overestimate, as already in the early applications of RMT, the validity of the independence condition was questioned by a number of works coming mainly from physics and related fields \cite{french1970validity, guhr1998random}. Unlike universality, very little is known about how much structure can one allow and still obtain an ensemble with the semicircular limiting law. Some recent works \cite{gotze2012semicircle, hofmann2008wigner, schenker2005semicircle} allow moderate amount of dependence between matrix elements. Usually, the level of permitted structure is limited by the specific tools used for the analysis and to the best of our knowledge no unified examination of the information-theoretic bounds of the possible dependencies has been performed so far. Moreover, due to the vagueness and complexity of determining the dependencies between the entries and their quantitative assessment, there does not seem to exist literature specifically concentrating on this issue and rigorously introducing the system of reference for such study. 

In this work, we will make an attempt towards a better understanding of the amount of structure that can be tolerated by the semicircular law. Below we resort to the Kolmogorov complexity paradigm to quantify the dependencies between the matrix elements for sign matrices. We give an explicit construction of matrices with substantial amount of structure whose spectra still converge to the semicircular law. Using general results from the theory of algorithmic complexity, we justify that the measure of structure existing in our construction is close to the maximal possible thus showing that indeed a large amount of dependence between the entries can be introduced without affecting the limiting law.

\subsection{Pseudo-random Sequences}
From a different perspective, in many engineering applications one needs to simulate random matrices. The most natural way to generate an instance of a random $n \times n$ sign matrix is to toss a fair coin $\frac{n(n+1)}{2}$ times, fill the upper triangular part of a matrix with the outcomes and reflect the upper triangular part into the lower. Unfortunately, for large $n$ such an approach would require a powerful source of randomness due to the independence condition \cite{gentle2013random}. In addition, when the data is generated by a truly random source, atypical  \textit{non-random looking} outcomes have non-zero probability of showing up. Yet another issue is that any experiment involving tossing a coin would be impossible to reproduce exactly. All these reasons stimulated researchers and engineers from different areas to seek for approaches of generating \textit{random-looking} data usually referred to as \textit{pseudo-random} sources or sequences of binary digits \cite{zepernick2013pseudo, golomb1967shift}. A wide spectrum of pseudo-random number generating algorithms have found applications in a large variety of fields including radar, digital signal processing, CDMA, coding theory, cryptographic systems, Monte Carlo simulations, navigation systems, scrambling, etc. \cite{zepernick2013pseudo}.

The term \textit{pseudo-random} is used to emphasize that the binary data at hand is indeed generated by an entirely deterministic causal process with low algorithmic complexity, but its statistical properties resemble some of the properties of data generated by tossing a fair coin. Remarkably, most efforts were focused on one dimensional pseudo-random sequences \cite{zepernick2013pseudo, golomb1967shift} due to their natural applications and to the relative simplicity of their analytical treatment. One of the most popular methods of generating pseudo-random sequences is due to Golomb \cite{golomb1967shift} and is based on linear-feedback shift registers capable of generating pseudo-random sequences of very low algorithmic complexity. The study of pseudo-random arrays and matrices was launched around the same time \cite{reed1962note, macwilliams1976pseudo, imai1977theory, sakata1981determining}. Among the known two dimensional pseudo-random constructions the most popular are the so-called perfect maps \cite{reed1962note, paterson1994perfect, etzion1988constructions}, and the two dimensional cyclic codes \cite{imai1977theory, sakata1981determining}. However, except for our recent article \cite{soloveychik2017pseudo} discussed below, to the best of our knowledge none of the previous works considered constructions of symmetric matrices using their spectral properties as the defining statistical features.

\subsection{The Kolmogorov Complexity}
There exist various approaches to quantify the algorithmic power needed to generate an individual piece of binary data, also known as algorithmic complexity \cite{grunwald2004shannon, li2009introduction, downey2010algorithmic}. It can be intuitively thought of as a measure of the amount of randomness stored in that piece of data. Probably the most popular among computer theorists measure of algorithmic complexity is the so-called Kolmogorov (Kolmogorov-Chaitin) complexity \cite{solomonoff1964formal1, kolmogorov1965three} defined as follows. Let $D$ be a string of binary data of length $n$, then its Kolmogorov complexity is the length of the shortest binary Turing machine code that can produce $D$ and halt. If $D$ has no computable regularity it cannot be encoded by a program shorter than its original length $n$ (here and below the Kolmogorov complexity is given up to an additive constant), meaning that its consecutive bits are unpredictable given the preceding ones, and it may be considered as truly random \cite{knuth1998art, li2009introduction}. A string with a regular pattern, on the other hand, can be computed by a program much shorter than the string itself, thus having a much smaller Kolmogorov complexity. By convention, a comparison of Kolmogorov complexities of various strings of the same length is usually done by conditioning on the length and thus assuming the length to be already known to the machine without specifying it as an input \cite{cover2012elements}. For example, the conditional Kolmogorov complexity of a Golomb sequence of length $n$ is at most $2\log_2 n$, which is relatively small, since using a simple combinatorial argument one can show that at most $\frac{n}{2^n}$ fraction of the strings of length $n$ have conditional Kolmogorov complexity less than $\log_2 n$.

We would like to emphasize that the definition of the Kolmogorov complexity does not require the data generating algorithm to be presented explicitly, and therefore is sometimes considered to be not very informative. There exist finer measures of algorithmic complexity taking into account the level of explicitness of the algorithm and/or its run time \cite{hoory2006expander}. The algorithm that we present in this paper is in fact ``very'' explicit even according to demanding definitions of complexity or explicitness. However, to avoid deep digression into the subtleties of different definitions of algorithmic complexity due to lack of space, and in order not to deviate much from the main topic of the article, in this work we focus on Kolmogorov complexity.


\subsection{Spectral Pseudo-randomness}
Most of the literature dealing with specific pseudo-random constructions start from a list of concrete properties mimicking truly random data and try to come up with a deterministic way of reproducing those properties. We would like to adapt the same paradigm and apply it to the property of random matrices to have asymptotically semicircular spectral distribution. It is important to emphasize that unlike the works listed above, we cannot require any single matrix to have semicircular spectral density, since the spectrum of a matrix is a step function and cannot be smooth. Therefore, we need to adjust the framework to our case and allow matrices to only approximately match the desired property, while requiring exact matching in the limit. This forces us to deal with sequences of matrices instead of single realizations. We concentrate on the problem of deterministically constructing sequences of symmetric sign matrices whose spectra converge to the semicircular law. A naturally arising question here may be formulated as: ``What is the minimal Kolmogorov complexity of matrices in a sequence of sign matrices, whose spectra converge to the semicircular law?''. Yet another important question that have been challenging mathematical and engineering society for the last decades is the \emph{inverse spectral problem} which can be formulated as: ``What can be said about a large sign matrix (an adjacency matrix of a graph lifted by the mapping $0 \rightarrow 1$ and $1 \rightarrow -1$) if its spectrum is known/is close to a semicircle?''. Based on the experience with Wigner's matrices and their spectral properties scientists tend to believe that if the spectrum of a matrix is close to the semicircle, it will necessary \emph{look random} and will also possess no observable structure. Unfortunately, there does not exist much literature on this problem due to its complexity. In this article, we try to shed some light on the aforementioned questions.

Using Golomb sequences (also known as binary $m$-sequences) of lengths $n=2^m-1$, we explicitly construct $n \times n$ symmetric circulant sign matrices of Kolmogorov complexity as low as $2\log_2 n$, whose spectra converge to the semicircular law with $n \to \infty$. The proof given below follows the classical method of moments and demonstrates that the empirical moments of a properly designed ensemble of our pseudo-random matrices converge almost surely to the correct limiting values, which implies almost surely weak convergences of the empirical distributions. This surprising result has at least three major consequences. First, it means that the real amount of randomness conveyed by the semicircular law is quite low. Second, it provides the first deterministic construction of such matrices, which may significantly affect many applications where random matrices are generated using more powerful sources of randomness. Finally, it partially answers the \emph{inverse spectral problem} by building a sequence of circulant structured matrices with spectra converging to Wigner's law, which contradicts the common belief that matrices with semicircular spectrum must be \emph{random looking}.

The proposed construction can be viewed as a pseudo-random analog of the truly random circulant model with independent Rademacher entries. The empirical spectra of the truly random circulant matrices converge almost surely to the normal law \cite{bose2002limiting} having unbounded support and, thus, being significantly different from the limiting distribution in our case. This means in particular that the intrinsic structure of the Golomb sequences combined with the circulant pattern surprisingly boils down into the semicircular law. This phenomenon requires further investigation and may contribute to the study of the \emph{inverse spectral problem}. The only other known pseudo-random symmetric sign matrices with spectra converging to the semicircular law are the elements of the pseudo-Wigner ensemble built from the dual BCH codes \cite{soloveychik2017pseudo}. Compared to that paper, our present model has a completely deterministic and easier construction algorithm and lower algorithmic complexity. In addition, in the present work convergence of the empirical spectra to the limiting law is shown to be almost sure, while in the setup of \cite{soloveychik2017pseudo} only convergence in probability can be guaranteed.

It can be easily observed that the second powers of our pseudo-random matrices provide a deterministic construction of matrices whose spectra converge to the Marchenko-Pastur law with the aspect ratio $\gamma=1$ \cite{marchenko1967distribution}. The Marchenko-Pastur distribution naturally arises as the limiting spectral density of high-dimensional sample covariance matrices with independent samples under some assumptions on the population distribution \cite{marchenko1967distribution}. The algorithmic complexity of the squared matrices is the same as of the original ones up to a constant term, thus yeilding an efficient pseudo-random construction of matrices with the limiting Marchenko-Pastur spectrum.

\subsection{Pseudo-random Graphs}
It is also instructive to relate our pseudo-random matrices to the numerous $\textit{random-looking}$ graphs mostly considered in combinatorics and the theory of computer science \cite{krivelevich2006pseudo, alon2004probabilistic}. Binary symmetric $n \times n$ matrices naturally correspond to unweighted, undirected graphs on $n$ vertices, since each matrix element $a_{ij}$ may be viewed as indicating the presence of the edge $(i,j)$. Graphs mimicking properties of truly random graphs are mainly utilized for the purposes of derandomization, see \cite{krivelevich2006pseudo, hoory2006expander, alon2004probabilistic} and references therein. Popular classes of $\textit{random-looking}$ graphs include pseudo-random (jumbled) \cite{thomason1987pseudo}, quasi-random \cite{chung1989quasi}, expander graphs \cite{hoory2006expander} and others. Remarkably, construction of some of them can be very involved (e.g. expanders), limiting their usage in high dimensions. In addition, in most pseudo-random graph constructions only results regarding their top eigenvalues are usually reported, leaving aside investigation of the bulk of their spectra.

The rest of the text is organized as follows. First we introduce notation in Section \ref{sec:not}. Section \ref{sec:gol} is devoted to the brief overview of Golomb sequences and linear-feedback shift registers. Then we define the notion of pseudo-random matrices constructed from Golomb sequences in Section \ref{sec:def} and formulate the main results of the article regarding the convergence of their empirical spectra to the semicircular distribution in Section \ref{sec:main}. In Section \ref{sec:kolm}, we discuss the Kolmogorov complexity of the proposed matrices. The numerical simulations and comparisons of our model with other related ensembles are given in Section \ref{sec:num_res}. All the auxiliary claims and the proofs of the main results are presented in Appendices.

\section{Notation}
\label{sec:not}
Let $\mathcal{C}$ be a $[n,k,d]$ binary linear code of length $n$, dimension $k$ and minimum Hamming distance $d$. The dual code $\mathcal{C}^\perp$ of $\mathcal{C}$ is a linear code of the same length and dimension $k^\perp = n-k$, whose codewords are orthogonal to all the codewords of $\mathcal{C}$, where we say that two words $\u = \{u_i\},\; \v = \{v_i\} \in GF(2)^n$ are orthogonal if $\sum_i v_iu_i = 0\; \mmod 2$. Introduce a family of functions
\begin{equation}
\label{eq:zeta_def}
\begin{array}{llcl}
\hspace{-0.25 cm} \zeta_n : & GF(2)^{n \times n} & \to & \{-1,1\}^{n \times n}, \\
& \{u_{ij}\}_{i,j=0}^{n-1} & \mapsto & \{(-1)^{u_{ij}}\}_{i,j=0}^{n-1},
\end{array}
\end{equation}
mapping binary $0/1$ matrices into sign matrices of the same sizes. Below we suppress the subscript and write $\zeta$ for simplicity.

We denote by $S_n$ the set of all symmetric $n \times n$ matrices with entries $\mypm \frac{1}{2\sqrt{n}}$. The Wigner ensemble $\mathcal{W}_n$ is defined as the set $S_n$ endowed with the uniform probability measure. Ranges of non-negative integers are denoted by $[n] = \{0,\dots,n-1\}$. Note also that the integer indexation of matrix elements starts with $0$.

We use the following standard notation for the limiting relations between functions. We write $f(n)=o(g(n))$ if $\lim_{n\to \infty} \frac{f(n)}{g(n)} = 0$ and $f(n)=O(g(n))$ if $|f(n)| \leqslant C |g(n)|$ for some constant $C$ and $n$ big enough.

\section{Golomb and Binary $m$-Sequences}
\label{sec:gol}
\subsection{Golomb Sequences}
\label{sec:gol_a}
Let $f(x)$ be a binary primitive polynomial of degree $m$ and let $\mathcal{C}$ be a cyclic code of length $n=2^m-1$ with the generating polynomial
\begin{equation}
h(x) = \frac{x^n-1}{f(x)}.
\end{equation}
In other words,
\begin{equation}
\mathcal{C} = \{\c \in GF(2)^n\mid h(x)|\c(x)\},
\end{equation}
where
\begin{equation}
\c(x) = \sum_{i=0}^{n-1}c_ix^i.
\end{equation}
When $f(x)$ is primitive, as in our case, a code constructed in such a way is usually referred to as a simplex code. All the non-zero codewords of the obtained code are shifts of each other and are called \textit{Golomb sequences} \cite{golomb1967shift} (we can, therefore, simply say that a simplex code is generated by a Golomb sequence). Note that the dual code of $\mathcal{C},\; \mathcal{C}^\perp$ is simply a binary Hamming code generated by $f(x)$. The dimensions of $\mathcal{C}$ and $\mathcal{C}^\perp$ are $m$ and $n-m$, correspondingly. 

It is sometimes convenient to glue a few copies of the same codeword together and consider them as a periodic sequence with the period at most $n$. Golomb sequences have many desirable statistical properties. The best known of these properties claim that Golomb sequences satisfy Golomb's axioms for pseudo-random sequences \cite{golomb1967shift}. 

Let $\varphi = \{\varphi(i)\}_{i=0}^{n-1}$ be a periodic sequence, then the three axioms for $\varphi$ to be a pseudo-random sequence are as follows.
\renewcommand{\labelenumi}{[G\theenumi]}
\begin{enumerate}
\item \label{ax1} \textbf{Distribution:}  In every period the number of ones is nearly equal to the number of zeros, more precisely the difference between the two numbers is at most 1
\begin{equation}
\left|\sum_{i=0}^{n-1} (-1)^{\varphi(i)}\right| \leqslant 1.
\end{equation}
\item \label{ax2} \textbf{Serial test I:} A sequence of consecutive ones is called a block and a sequence of consecutive zeros is called a gap. A run is either a block or a gap. In every period, one half of the runs has length $1$, one quarter of the runs has length $2$, and so on, as long as the number of runs indicated by these fractions is greater than $1$. Moreover, for each of these lengths the number of blocks is equal to the number of gaps.
\item \label{ax3} \textbf{Auto-correlation:} The auto-correlation function
\begin{equation}
C(a) = \sum_{i=0}^{n-1} (-1)^{\varphi(i)+\varphi(i+a)}
\end{equation}
is two-valued.
\end{enumerate}

The distribution axiom [G\ref{ax1}] is a special case of the serial test axiom [G\ref{ax2}]. However, [G\ref{ax1}] is retained for historical reasons, and sequences which satisfy [G\ref{ax1}] and [G\ref{ax3}], but not [G\ref{ax2}], are also important.
\begin{lemma}[\cite{golomb1967shift}]
Every Golomb sequence satisfies axioms [G\ref{ax1}]-[G\ref{ax3}].
\end{lemma}

Besides the listed Golomb axioms, Golomb sequences also possess other ubiquitous properties.
\renewcommand{\labelenumi}{[P\theenumi]}
\begin{enumerate}
\setcounter{enumi}{3}
\item \textbf{Shift:} since $\mathcal{C}$ is a cyclic code, a cyclic shift of a Golomb sequence is a Golomb sequence. In addition, $\mathcal{C}$ is generated by any one of them.
\item \label{prop:s-a-a}\textbf{Shift-and-add:} for any $1 \leqslant a \leqslant n-1$ there exists a unique $1 \leqslant b \leqslant n-1$, such that
\begin{equation}
\varphi(i) + \varphi(i+a) = \varphi(i+b)\;\;\; \mmod 2, \;\; \forall i.
\end{equation}
Remarkably, a binary sequence satisfies the shift-and-add property iff it is a Golomb sequence.
\item \label{prop:rec}\textbf{Recurrence:} suppose 
\begin{equation}
\label{eq:rec_rel}
f(x) = \sum_{j=0}^m f_jx^j,\;\; f_j \in GF(2),\;\; f_0=f_m=1,
\end{equation}
then any Golomb sequence $\varphi$ generated by $f(x)$ satisfies the recurrence relation
\begin{equation}
\label{eq:gol_seq_gen}
\varphi(i+m) = \sum_{j=0}^{m-1} f_j \varphi(i+j),\;\;\; \mmod 2\;\; \forall i.
\end{equation}
Using all $n=2^m-1$ distinct non-zero initial values $\varphi(0),\dots,\varphi(m-1)$ in (\ref{eq:rec_rel}), we obtain $n$ Golomb sequences, which are all cyclic shifts of one sequence. There are $m$ linearly independent solutions to (\ref{eq:rec_rel}), hence $m$ linearly independent (over $GF(2)$) Golomb sequences generated by one polynomial $f(x)$.
\item \textbf{Window:} if a window of width $m$ is slid along a Golomb sequence, each of the $n$ different non-zero binary $m$-tuples is seen exactly once. This follows from the fact that $f(x)$ is a primitive polynomial.
\item \textbf{Serial test II:}  For any binary $k$-tuple $\b$, let $M(\b)$ denote the number of occurrences of $\b$ in one period of $\varphi$. Then for any $k$ with $1\leqslant k \leqslant \log_2 n$ and for any two $k$-tuples $\b$ and $\c$, we have
\begin{equation}
\left|M(\b)-M(\c)\right| \leqslant 1.
\end{equation}
\end{enumerate}

\subsection{Linear-Feedback Shift Registers}
\begin{figure}[t]
\centering
\includegraphics[width=3in]{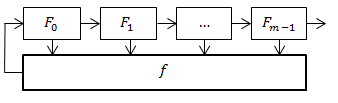}
\caption{Linear-feedback shift register defined by $f(x)$.}
\label{fig:lfsr}
\end{figure}
Property [P\ref{prop:rec}] suggests a very efficient way of generating Golomb sequences using the so-called Linear-Feedback Shift Registers (LFSR) and a timer. An LFSR is a shift register, that is a sequence of flip-flops $F_0, \dots, F_{m-1}$ as in Figure \ref{fig:lfsr} storing bits. At each tick of the clock, $F_i$ takes the value of $F_{i-1}$ for $i > 0$ and $F_0$ is updated according to the Boolean linear feedback function $f$, generating the output sequence
\begin{equation}
\label{eq:lsr_rec}
\phi(i+m) = \sum_{j=0}^{m-1} f_j \phi(i+j)\; \mod 2,\;\; \forall i.
\end{equation}
With $m$ flip-flops, we can realize a machine with up to $2^m$ states. We will always assume that the value of $\phi_{m-1}$ is the output of the shift register.

\begin{definition}
A LFSR sequence is a sequence $\phi$ satisfying the recursion (\ref{eq:lsr_rec}).
\end{definition}

Since the next value depends only on the preceding $m$ values, the sequence must be periodic. The state $(\phi(i),\dots,\phi(m+i)) = (0,\dots,0)$ leads to the constant zero sequence, thus the period of an LFSR sequence over $GF(2)$ cannot exceed $2^m-1$.

\begin{definition}
A binary LFSR sequence with period $2^m-1$ is called an $m$-sequence (maximal sequence).
\end{definition}

\begin{lemma}[\cite{golomb1967shift}]
\label{lem:lfsr_lem}
All Golomb sequences generated from different primitive polynomials and only they are the binary $m$-sequences.
\end{lemma}
Below we use the pseudo-random properties of Golomb sequences listed above to construct pseudo-random matrices. Interestingly, it was shown in \cite{zierler1959linear} that the Shift-and-add Property [P\ref{prop:s-a-a}] is the crucial one as it completely determines the binary $m$-sequences and imply all their other properties.

\section{The Construction}
\label{sec:def}
In this section, we provide a certain construction of a sequence of $n \times n$ matrices with $n=2^m-1,\; m \in \N$ from Golomb sequences of lengths $n$. Later we show that when the sizes of these matrices grow, their spectra converge to the semicircular law.

Let $\mathcal{C}$ be the simplex code constructed from the primitive binary polynomial $f(x)$ as before. Fix a non-zero codeword $\varphi \in \mathcal{C}$ (a Golomb sequence) and construct a real symmetric matrix
\begin{equation}
\label{eq:def_a_eq}
\A_n = \{a_{ij}\}_{i,j=0}^{n-1} = \bigg\{\frac{1}{2\sqrt{n}}(-1)^{\varphi(i-j) + \varphi(j-i)}\bigg\}_{i,j=0}^{n-1}.
\end{equation}
Matrix $\A_n$ can be interpreted in the following way. Consider a circulant non-symmetric matrix

\begin{equation}
\T = \begin{pmatrix} 
\varphi(0) & \varphi(1) & \varphi(2) & \dots & \varphi(n-1) \\
\varphi(n-1) & \varphi(0) & \varphi(1) & \dots & \varphi(n-2) \\
\varphi(n-2) & \varphi(n-1) & \varphi(0) & \dots & \varphi(n-3) \\
\vdots & \vdots & \vdots & \ddots & \vdots \\
\varphi(1) & \varphi(2) & \varphi(3) & \dots & \varphi(0) \\
\end{pmatrix}.
\end{equation}
The consecutive rows of $\T$ are simply cyclic shifts of the Golomb sequence written in its first rows. The symmetric matrix $\A_n$ can now be written as
\begin{equation}
\A_n = \frac{1}{2\sqrt{n}}\zeta(\T+\T^\top).
\end{equation}
It is easy to check that the obtained matrix is circulant, since for any $k \in [n],\; \A_n$ is invariant under the shift of indices of the form
\begin{equation}
i \to i+ k \mod n,\quad j \to j+k \mod n.
\end{equation}

Recall that any non-zero codeword of $\mathcal{C}$ is a cyclic shift of $\varphi$, therefore, we may obtain an ensemble of matrices from the code $\mathcal{C}$ indexed by integers within the range $a \in [n]$, as
\begin{equation}
\A_n(a) = \bigg\{\frac{1}{2\sqrt{n}}(-1)^{\varphi(i-j+a) + \varphi(j-i+a)}\bigg\}_{i,j=0}^{n-1},
\end{equation}
with the original matrix $\A_n$ corresponding to $\A_n(0)$.
\begin{definition}
Given a primitive binary polynomial $f(x)$, an ensemble of pseudo-random matrices $\mathcal{A}_n$ of order $n$ is the set of all $\A_n(a),\; a \in [n]$ and their negatives, endowed with the uniform probability measure.
\end{definition}
Below, whenever expectation over $\A_n$ is considered it should be always treated with respect to the uniform measure over $\mathcal{A}_n$.

\section{The Spectra of the Pseudo-random Matrices}
\label{sec:main}
In this section, we demonstrate that the spectra of matrices $\A_n$ uniformly chosen from $\mathcal{A}_n$ converge almost surely to the semicircular law when their sizes grow to infinity. We start from a number of auxiliary definitions and results. 

\subsection{Wigner's Ensemble and the Semicircular Law}
For a real symmetric matrix $\A_n \in S_n$, denote by
\begin{equation}
\lambda_1(\A_n) \leqslant \dots \leqslant \lambda_n(\A_n)
\end{equation}
its eigenvalues, which are all real. Let $F_{\A_n}$ be the cumulative density function (c.d.f.) associated with the spectrum of $\A_n$,
\begin{equation}
F_{\A_n}(x) = \frac{1}{n}\sum_{i=1}^{n} \theta\(x-\lambda_i(\A_n)\),
\end{equation}
where $\theta(x)$ is the unit step function at zero. The $r$-th moment of $\A_n$ reads as
\begin{equation}
\label{eq:mom_def_trace}
\beta_r(\A_n) = \int x^r dF_{\A_n} = \frac{1}{n}\Tr{\A_n^r}.
\end{equation}
Denote by $F_{sc}$ the c.d.f. of the standard semicircular law
\begin{multline}
\label{eq:sc_cdf_def}
F_{sc}(x) = \frac{1}{2} + \frac{1}{\pi}x\sqrt{1-x^2}+\frac{1}{\pi}\arcsin(x), \\ -1 \leqslant x \leqslant 1,
\end{multline}
and the corresponding probability density function (p.d.f.) by
\begin{equation}
\label{eq:sc_pdf_def}
f_{sc}(x) = \frac{2}{\pi} \sqrt{1-x^2},\;\; -1 \leqslant x \leqslant 1.
\end{equation}
The moments of this distribution read as \cite{wigner1955characteristic}
\begin{equation}
\label{eq:dc_mom}
\beta_r = \int x^r dF_{sc} = \begin{cases} \;\; 0, & r \text{ odd}, \\ \frac{C_{r/2}}{2^r}, & r \text{ even},\end{cases}
\end{equation}
where
\begin{equation}
\label{eq:catal_num}
C_r = \frac{(2r)!}{r!(r+1)!}
\end{equation}
is the $r$-th Catalan number. When the sizes of Wigner matrices grow, their moments converge to those of the semicircular law and their variances are summable, which ensures the almost sure convergence to the semicircular law as stated by the following
\begin{lemma}[\cite{wigner1955characteristic, arnold1967asymptotic}, Theorem 2.5 from \cite{bai2010spectral}]
\label{lem:mom_conv_norm}
Let $\W_n \in \mathcal{W}_n$, then as $n$ tends to infinity,
\begin{equation}
\mathbb{E}\[\beta_r(\W_n)\] = \begin{cases} \beta_r + o(1), & r \text{ even}, \\ \quad\;\; 0, & r \text{ odd},\end{cases}
\end{equation}
and for the variances of $\beta_r(\W_n)$,
\begin{equation}
\label{eq:var_cond}
\var{\beta_r(\W_n)} = O\(\frac{1}{n^2}\).
\end{equation}
\end{lemma}

\begin{corollary}[\cite{arnold1967asymptotic}, Theorem 2.5 from \cite{bai2010spectral}]
\label{corol:conv}
The empirical spectra of matrices from Wigner's ensemble converge almost surely weakly to the semicircular law.
\end{corollary}
\begin{proof}
Follows from Lemma \ref{lem:mom_conv_norm} using Chebyshev's inequality and the Borel-Cantelli Lemma.
\end{proof}

\subsection{Spectrum of the Pseudo-random Construction}
In the next two lemmas, we show an analog of Lemma \ref{lem:mom_conv_norm} for the proposed pseudo-random matrices which would guarantee the almost sure weak convergence of their spectra to the semicircular law.
\begin{prop}
\label{lem:main_exp}
Let $\A_n \in \mathcal{A}_n$, then for a fixed $r \in \N$ and $n = n(m)$ tending to infinity,
\begin{equation}
\mathbb{E}\[\beta_r(\A_n)\] = \begin{cases} \beta_r + O\(\frac{1}{n}\), & r \text{ even}, \\ \qquad 0, & r \text{ odd}.\end{cases}
\end{equation}
\end{prop}
\begin{proof}
The proof can be found in Appendix \ref{app:proof_main_res}.
\end{proof}

\begin{prop}
\label{lem:main_var}
Let $\A_n \in \mathcal{A}_n$, then for a fixed $r \in \N$ and $n = n(m)$ tending to infinity,
\begin{equation}
\var{\beta_r(\A_n)} = O\(\frac{1}{n^2}\).
\end{equation}
\end{prop}
\begin{proof}
The proof can be found in Appendix \ref{app:proof_main_res}.
\end{proof}

\begin{corollary}
\label{cor:main}
The empirical spectra of matrices $\A_n \in \mathcal{A}_n$ converge almost surely weakly to the semicircular law.
\end{corollary}
\begin{proof}
Follows from Propositions \ref{lem:main_exp} and \ref{lem:main_var} through the same argument as in the proof of Corollary \ref{corol:conv}.
\end{proof}

\begin{rem}
In fact, we believe that even a stronger result holds. Namely, that the empirical spectra of matrices $\A_n$ converge to the semicircular law when $n$ grows. This statement does not involve probability and averaging over an ensemble, which may make it harder to be proved.
\end{rem}

This surprising result has a number of consequences. First, it provides a deterministic construction of matrices whose spectra converge to the semicircular law and shows that such matrices may be constructed with the help of a simple algorithm. In Section \ref{sec:kolm}, we exactly quantify the algorithmic complexity of the proposed construction. Second, unlike the common belief that matrices whose spectra converge to the semicircular law must be \emph{random-looking}, it shows that this is not the case and they may in fact be very structured. In particular, the pseudo-random family of matrices we suggest consists of circulant matrices commuting with each other and, therefore, having a common fixed eigenbasis (this basis can be easily obtained from a discrete Fourier matrix).

\subsection{Relation to the Marchenko-Pastur Law}
\label{sec:mp_subseq}
Interestingly, the above construction also enables us to build a sequence of pseudo-random matrices whose empirical spectral distributions converge to the Marchenko-Pasur law defined through its p.d.f. as
\begin{align}
\label{eq:mp_law}
f_{MP}(x;\gamma) = \frac{1}{2\pi \gamma x}\sqrt{(a_+-x)(x-a_-)}, & \\
& \hspace{-1.2cm} a_- \leqslant x \leqslant a_+, \nonumber
\end{align}
where $a_{\pm} = (1\pm\sqrt{\gamma})^2$ and $\gamma$ is the so-called aspect ratio parameter. Distribution (\ref{eq:mp_law}) naturally appears as the limiting spectral law of the sample covariance matrices of $m_n$ i.i.d. (independent and identically distributed) isotropic $n$-dimensional samples with
\begin{equation}
\gamma = \lim_{n \to \infty} \frac{m_n}{n} \leqslant 1,
\end{equation}
under some additional conditions on the population distribution \cite{marchenko1967distribution, pillai2014universality}. As a direct corollary of the main result from the seminal work \cite{marchenko1967distribution} of Marchenko and Pastur we get the following statement.
\begin{lemma}[Corollary of Theorem 1 from \cite{marchenko1967distribution}]
\label{prop:mp_law}
Let for every $n \in \N,\; \{\x_{ij}\}_{i,j=0}^{m_n-1,n-1}$ be i.i.d. Rademacher random variables and denote
\begin{equation}
\X_{m_n,n} = \bigg\{\frac{1}{\sqrt{n}}x_{ij}\bigg\}_{i,j=0}^{m_n-1,n-1},
\end{equation}
then the empirical spectra of the sample covariance matrices $\X_{m_n,n}\X_{m_n,n}^\top$ converge almost surely weakly to the Marchenko-Pastur law with aspect ratio $\gamma$.
\end{lemma}

As a corollary of Propositions \ref{lem:main_exp} and \ref{lem:main_var}, we get
\begin{corollary}
\label{cor:main_mp}
The empirical spectra of matrices $4\A_n^2$ with $\A_n \in \mathcal{A}_n$ converge almost surely weakly to the Marchenko-Pastur law with $\gamma=1$.
\end{corollary}
\begin{proof}
The moments of the Marchenko-Pastur distribution are given by the formula
\begin{equation}
\label{eq:dc_mom}
\eta_r(\gamma) = \int x^r f_{MP}(x;\gamma) dx = \sum_{k=1}^r \gamma^k N(r,k),
\end{equation}
where
\begin{equation}
N(r,k) = \frac{1}{r}{k \choose r}{k-1 \choose r}
\end{equation}
are Narayana numbers. Use the identity
\begin{equation}
C_r = \sum_{k=1}^r N(r,k)
\end{equation}
relating Narayana and Catalan numbers to conclude that for $\gamma = 1$,
\begin{equation}
\label{eq:dc_mom}
\eta_r(1) = \sum_{k=1}^r N(r,k) = C_r.
\end{equation}

The $r$-th moment of $4\A_n^2$ is exactly the $2r$-th moment of $2\A_n$ and since the latter converges almost surely to the Catalan number $C_r$ as follows from Propositions \ref{lem:main_exp} and \ref{lem:main_var}, we get the desired statement.
\end{proof}

Similarly to Wigner's case, this surprising result suggests that even in cases where the measurements at hand (the columns of $\X_{m_n,n}$) are significantly dependent, the limiting distribution may still remain to be the Marchenko-Pastur law. In our current work, we aim at designing deterministic examples of matrices whose spectra converge to the Marchenko-Pastur distribution with arbitrary aspect ratio.

\section{Algorithmic Complexity of the Construction}
\label{sec:kolm}
The standard computer scientific approach to quantify the amount of randomness contained in a piece of data $D$, also known as its algorithmic compressibility, is based on calculating the length of a minimal program creating that data on a universal Turing machine. The length of the obtained binary code is referred to as the Kolmogorov complexity of the object and we denote it by $\mathcal{KC}(D)$. A more fair comparison of Kolmogorov complexities of various objects of the same size is usually done by conditioning on that size $\mathcal{KC}(D|\text{size})$, or in other words by assuming that it is already known to the machine \cite{cover2012elements}.

One may argue that measuring randomness of data samples through the concept of Kolmogorov complexity is not particularly useful since Kolmogorov complexity is not computable even using infinite computing resources. Indeed, it is a consequence of the fact that the halting problem for Turing machines is undecidable \cite{hopcroft2006automata}, that it is theoretically impossible, not only computationally infeasible, to check all possible generation algorithms for a given piece of data and to choose the shortest among them. To bypass this problem researchers and engineers often resort to the notion of linear complexity, which similarly to Kolmogorov complexity seeks for the shortest program, however, restricts the search to the class of LFSR-s. This restriction of the domain of search makes linear complexity easily computable thanks to the Berlekamp-Massey algorithm \cite{berlekamp1968algebraic,massey1969shift}, which can find the shortest LFSR generating a given binary string. Apparently, linear complexity always upper bounds the Kolmogorov complexity.

Since we focus on circulant matrices, the whole construction essentially reduces to the description of the first row. Due to equation (\ref{eq:def_a_eq}), the first row of a pseudo-random matrix at hand is constructed from a single Golomb sequence. Therefore, up to a constant length code describing the pattern of the matrix, the algorithmic complexity of a pseudo-random $n \times n$ matrix is that of a Golomb sequence $\varphi$ of length $n$. The linear complexity of a Golomb sequence is easily computable since by Lemma \ref{lem:lfsr_lem} they are exactly the binary $m$-sequences. To generate a Golomb sequence of length $n=2^m-1$ we need to specify the associated primitive binary polynomial of degree $m$, which requires at most $m-1$ bits (recall that its $m$-th and $0$-th coefficients are always $1$) and the initial state of the register, the so-called \emph{seed}, of $m$ bits. Overall, the linear complexity of a Golomb sequence is at most $2m-1 = 2\log_2 (n+1)-1$, which is asymptotically equivalent to
\begin{equation}
2\log_2 (n+1)-1 \asymp 2\log_2 n.
\end{equation}
Thus, we have proven the following result.
\begin{prop}
The Kolmogorov complexity of the pseudo-random matrices from (\ref{eq:def_a_eq}) is bounded by
\begin{multline}
\mathcal{KC}(\A_n|n) = \mathcal{KC}(\varphi|n) +c \\ \leqslant 2\log_2 (n+1) - 1 + c \asymp 2\log_2 n.
\end{multline}
\end{prop}

Summarizing, Corollary \ref{cor:main} demonstrates that there exist matrices with empirical spectra converging to the semicircular law whose Kolmogorov complexity is asymptotically as low as $2\log_2 n$. To get an intuition of how small this value is, it is instructive to compare it with the following classical result.
\begin{lemma}[\cite{li2009introduction}]
At most $\frac{n}{2^n}$ fraction of all binary strings of length $n$ have conditional Kolmogorov complexity less than $\log_2 n$.
\end{lemma}
This allows us to claim that the randomness requirements for the semicircular law are comparatively small. In Section \ref{sec:pseudo-Wigner_matr}, we compare the proposed construction to pseudo-Wigner matrices defined in \cite{soloveychik2017pseudo}.

It is important to emphasize that the Kolmogorov complexity alone does not entirely capture the power of a sequence of matrices to approximate the limiting spectral law. Indeed, given a sequence of matrices we can always \textit{dilute} it by inserting other matrices between the elements to get a sequence with lower algorithmic complexity but worse convergence properties. Therefore, a more objective comparison should consider the Kolmogorov complexity versus the speed of convergence of the empirical measure to the limiting law in some metric. For example, let us consider the variance of the approximation as such measure of error. As justified in Remarks \ref{rem:better_conv_var} and \ref{rem:better_conv_var_pr} in Appendix \ref{app:proof_main_res}, in both Wigner's original case and our pseudo-random construction, the decay of the variance is of the same order $\frac{1}{n^2}$ as a function of the size. However, the Kolmogorov complexity of a typical matrix from Wigner's ensemble is equal to $\frac{n(n+1)}{2}$ up to an additive constant \cite{li2009introduction}, whereas the complexity of our construction is at most $2\log_2 n$.

\section{Numerical Simulations and Related Models}
\label{sec:num_res}
In this section we justify the above construction numerically and compare it to related constructions, both random and pseudo-random.

\subsection{Numerical Experiments}
\label{sec:num_constr}
In our first numerical experiment we took two Golomb sequences for $m=13\; (n=8191)$ and $m=14\; (n=16383)$ generated through relation (\ref{eq:gol_seq_gen}) by the primitive binary polynomials $f_{13}(x) = x^{13}+x^8+x^5+x^3+1$ and $f_{14}(x) = x^{14}+x^{12}+x^{11}+x+1$, correspondingly \cite{vzivkovic1994table}, with the initial seed being a sequence $[1,\dots,1]$ of $m$ ones. The matrices in this simulation were built from $100$ random shifts of the obtained sequences through formula (\ref{eq:def_a_eq}). Figures \ref{fig:dens_13} and \ref{fig:dens_14} show the empirical density functions of the spectra of the matrices. For convenience, we also plot the standard semicircular density and the average empirical distributions on the same graphs. The pictures clearly demonstrate that the empirical distributions and their averages become closer to the limiting semicircular shape when the sizes of matrices grow. In particular, we can observe that the width of the blue region shrinks, in accordance with the theoretically predicted by Proposition \ref{lem:main_var} dependence of the variance on the dimension, which implies that the standard deviations of the moments should be inverse proportional to $n$.

\subsection{Circulant Random Matrices}
By the very construction, our pseudo-random matrices have circulant pattern, therefore, probably the closest relative of the model at hand among truly random matrices is the ensemble of symmetric circulant random matrices with i.i.d. (up to symmetry constraints) Rademacher entries. The authors of \cite{bose2002limiting} demonstrated that the empirical spectra of such matrices after appropriate scaling converge almost surely weakly, as the sizes of matrices grow to infinity, to the normal distribution. Given this result and the fact that the Golomb sequences are commonly considered and used as substitutes for truly random Rademacher sequences, it is quite mysterious that the spectra of our pseudo-random matrices converge to a different distribution, and even more surprisingly the latter is the semicircular distribution having bounded support. To illustrate the phenomenon, in Figure \ref{fig:dens_tt} we plot the empirical spectra of $100$ symmetric truly random circulant matrices with $n=4000$ Rademacher entries scaled by the factor of $\frac{1}{2\sqrt{n}}$, their average spectrum, and the limiting spectrum. 

\label{sec:nsim}
\begin{figure}[!t]
\centering
\includegraphics[width=3.75in]{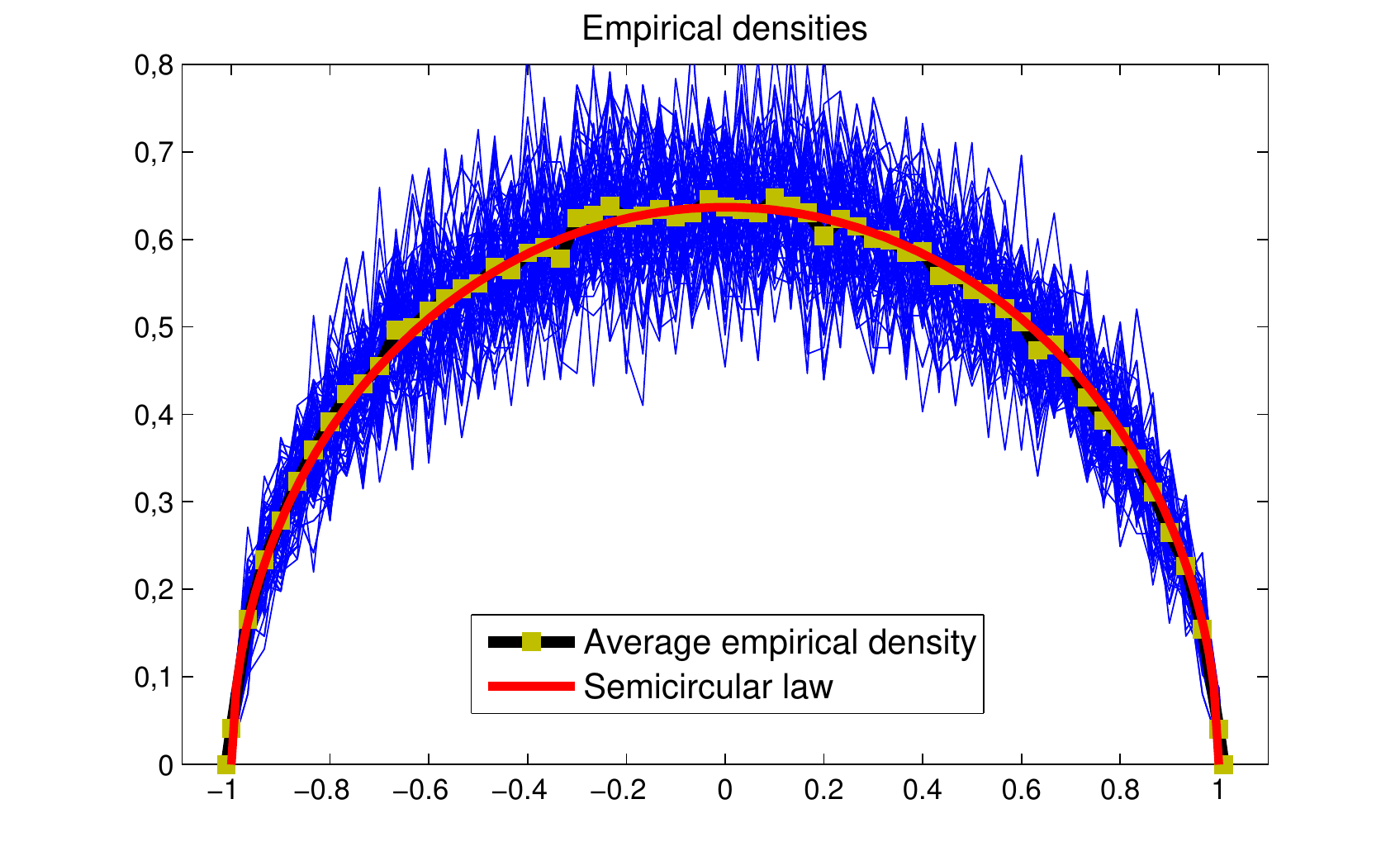}
\vspace{-0.6cm}
\caption{Empirical spectral densities of $100$ pseudo-random matrices, $m=13\; (n=8191)$.}
\label{fig:dens_13}
\vspace{0.5cm}
\includegraphics[width=3.75in]{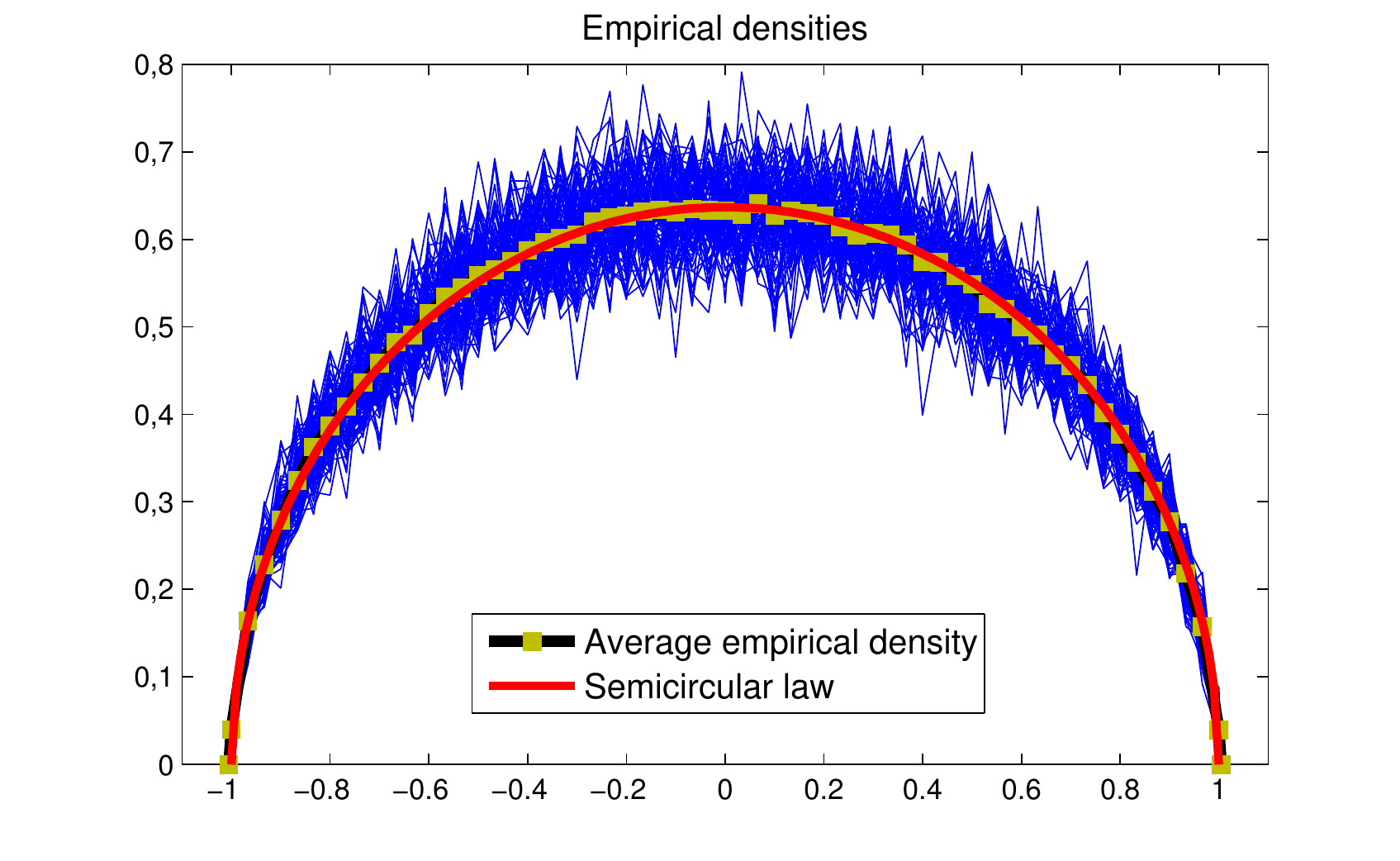}
\vspace{-0.6cm}
\caption{Empirical spectral densities of $100$ pseudo-random matrices, $m=14\; (n=16383)$.}
\label{fig:dens_14}
\end{figure}

\begin{figure}[!t]
\centering
\includegraphics[width=3.75in]{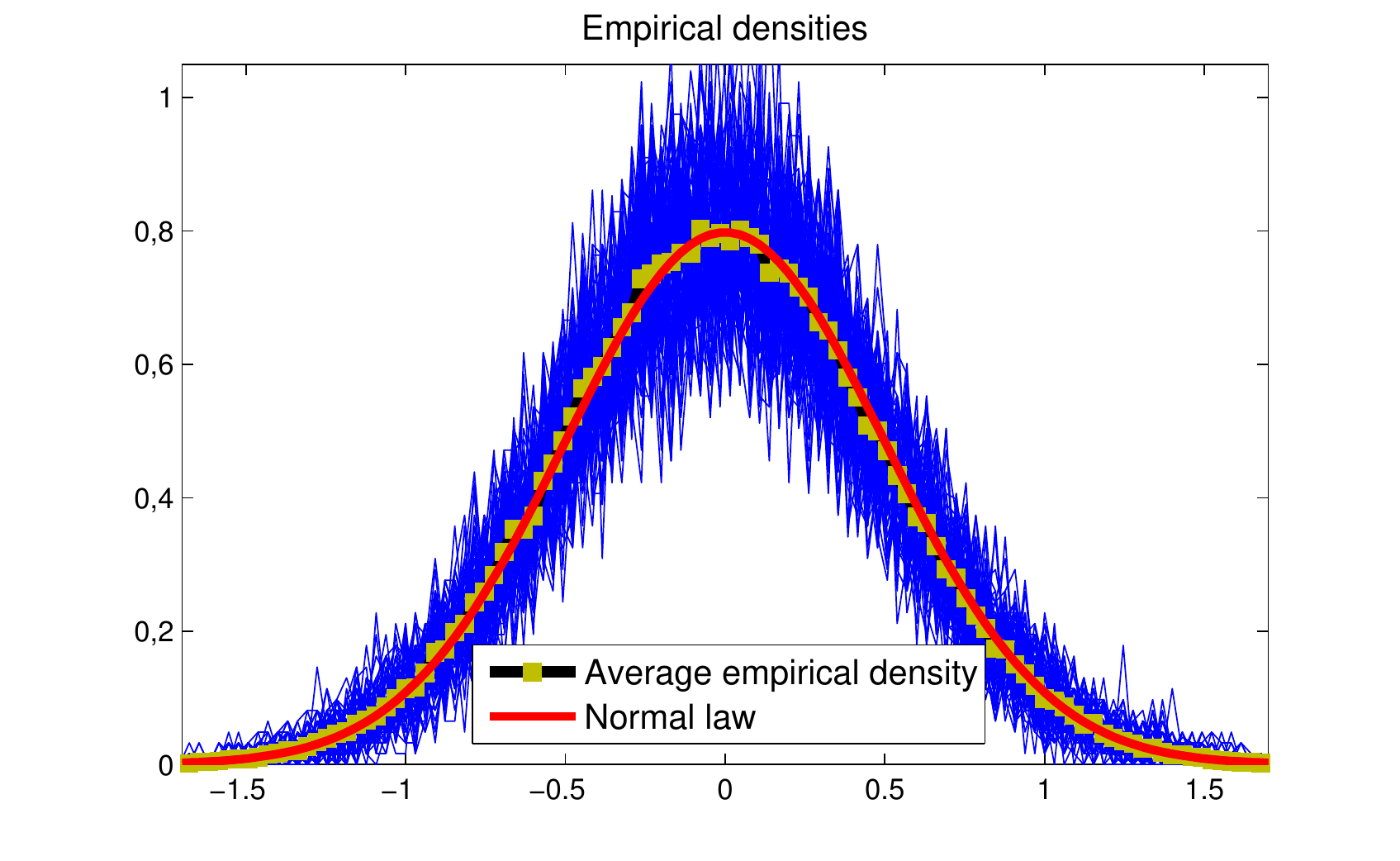}
\vspace{-0.6cm}
\caption{Empirical spectral densities of $100$ random circulant matrices, $n=4000$.}
\label{fig:dens_tt}
\vspace{0.5cm}
\includegraphics[width=3.75in]{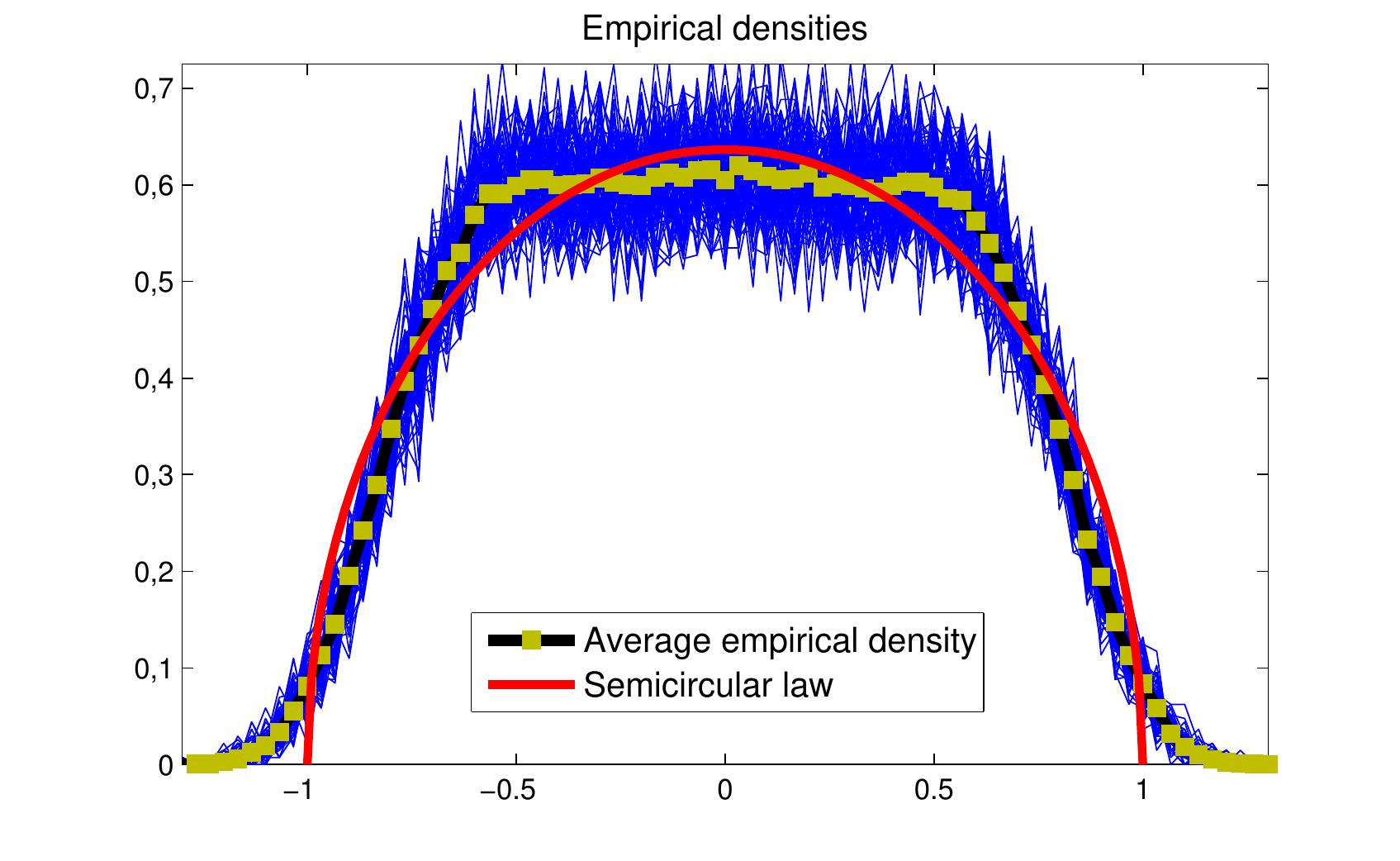}
\vspace{-0.6cm}
\caption{Empirical spectral densities of $100$ matrices $\D_n$ defined in (\ref{eq:dn_def}), $m=13\; (n=8191)$.}
\label{fig:dens_II13}
\end{figure} 

The above observation suggests that the Golomb sequences combined with the circulant pattern produce semicircle law instead of the anticipated normal distribution. This means that the intrinsic structure of the Golomb sequences is somehow related to the circulant structure. One of the current directions of our research is to reveal this connection and to try to use if for the construction of matrix ensembles with pre-designed limiting spectrum, which we called an \emph{inverse spectral problem} in the Introduction.

\begin{figure*}[!t]
\centering
\includegraphics[height=7.8cm]{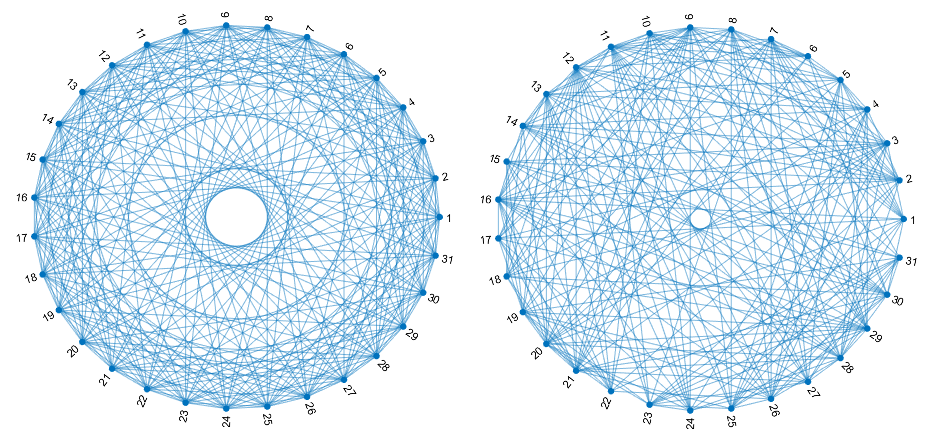}
\caption{A graph associated to the pseudo-random matrix vs a truly random graph, $m=5\;\;\; (n=31)$.}
\label{fig:rand_vs_prand}
\end{figure*} 

\subsection{Tridiagonal Semicircular Matrices}
Another known matricial model with semicircular limiting spectral shape and comparatively low algorithmic complexity is the ``tridiagonalized'' $\beta$-Hermite (Gaussian) ensemble introduced in \cite{dumitriu2002matrix}, with $\beta=1$. As suggested by the name, the matrices from this ensemble are tridiagonal and have spectral distributions identical to those of the matrices from the Gaussian orthogonally invariant ensemble, whose limiting law is semicircular. In the construction proposed by the authors of \cite{dumitriu2002matrix}, the entries on the main diagonal are i.i.d. centered normal with variance $2$ and on the first (upper) subdiagonal - i.i.d. $\chi^2(n-1)$ random variables. The spectra of such matrices scaled by $\frac{1}{2\sqrt{n}}$ converge to the semicircular law and the Kolmogorov complexity of this construction is proportional to $n$ (if the non-zero entries are stored with fixed precision). 

Importantly, the ``tridiagonalized'' Gaussian model does not possess the universality property (meaning that the distributions chosen for the non-zero entries cannot be replaced by any other laws possessing the same first moments) and does not allow a Rademacher analog with sign entries. It is yet an open question whether narrow-banded matrices with sign entries may have spectra converging to the semicircular law.

\subsection{Pseudo-Wigner Matrices}
\label{sec:pseudo-Wigner_matr}
Both previous examples treated constructions in which the (non-zero) entries of the matrices were generated independently. An example of pseudo-random construction proposed in \cite{soloveychik2017pseudo} suggests to build matrices from the codewords of dual BCH codes with appropriately chosen designed minimum distances of the original BCH codes. The purpose of the specific tuning of the minimum distance there is to make the first expected moments of the constructed matrices match those of the truly random Wigner matrices. The suggested technique explains the ``pseudo-Wigner'' title of the resulting ensembles. As shown in \cite{soloveychik2017pseudo}, when the number of matching moments grows at a specific rate with the sizes of matrices, the empirical spectra of the individual matrices from the pseudo-Wigner ensemble converge with high probability to the semicircle. 

There are a few significant differences between the pseudo-Wigner model \cite{soloveychik2017pseudo} and our approach in this paper. First, we guarantee convergence almost surely, which in practice means that any matrix from the pseudo-random ensemble converges to the semicircle, unlike the pseudo-Wigner matrices, where the convergence is in probability. Another important discrepancy is that the generating process of the dual BCH codes is somewhat more involved and the cardinality of the resulting ensemble is much larger forcing the Kolmogorov complexity of a single element to be higher. The Kolmogorov complexity of an instance from the pseudo-Wigner ensemble in \cite{soloveychik2017pseudo} is bounded from above by $\frac{2}{\varepsilon}\log_2 n$, where $\varepsilon > 0$ is a precision parameter. Since we want to have convergence, we need to decrease $\varepsilon = \varepsilon(n)$ when $n$ grows, which means that the ratio 
\begin{equation}
\frac{\frac{2}{\varepsilon(n)}\log_2 n}{2 \log_2 n} = \frac{1}{\varepsilon(n)}
\end{equation}
must converge to infinity and hence, the complexity of the pseudo-Wigner matrices is higher than that of matrices built here from the Golomb sequences.

An interesting insight on the difference between the Kolmogorov complexities of these two constructions can be obtained from the following observation. The construction presented in this article leads to the circulant pattern of the pseudo-random matrices. As mentioned above, this means that for a fixed $n$, the eigenbasis of the family of matrices at hand does not depend on the choice of the underlying Golomb sequence $\varphi$ and on the specific shift $a$. Therefore, the whole amount of randomness possessed by the model is spent on the construction of the spectrum, unlike the pseudo-Wigner matrices, where some amount of \emph{randomness} is consumed by the not-aligned eigenvectors. This is a purely intuitive understanding which we would not make precise in the current paper but rather postpone to our future studies.

\subsection{A Related Family of Golomb Matrices}
Yet, another natural and similar way to fill a circulant matrix using a Golomb sequence $\varphi$ is as follows. Let a real symmetric matrix $\D_n$ have the form
\begin{equation}
\label{eq:dn_def}
\D_n = \bigg\{\frac{1}{2\sqrt{n}}(-1)^{\varphi(|i-j|)}\bigg\}_{i,j=0}^{n-1},
\end{equation}
or
\begin{equation}
\D_n = \frac{1}{2\sqrt{n}}\zeta(\T),
\end{equation}
with
{\small \begin{equation*}
\T = \begin{pmatrix} 
\varphi(0) & \varphi(1) & \varphi(2) & \dots & \varphi(n-1) \\
\varphi(1) & \varphi(0) & \varphi(1) & \dots & \varphi(n-2) \\
\varphi(2) & \varphi(1) & \varphi(0) & \dots & \varphi(n-3) \\
\vdots & \vdots & \vdots & \ddots & \vdots \\
\varphi(n-1) & \varphi(n-2) & \varphi(n-3) & \dots & \varphi(0) \\
\end{pmatrix}.
\end{equation*}}
In spite of the fact that this construction resembles the one of Section \ref{sec:def}, the limiting spectral law of such matrices is not semicircular, as Figure \ref{fig:dens_II13} demonstrates. For this simulation we took $100$ random shifts of the same sequence as in Section \ref{sec:num_constr} above. This example demonstrates that when the sequence $\varphi$ is reflected in (\ref{eq:def_a_eq}) and added to itself it has a significant influence on the outcome as it couples with the circulant structure to produce the semicircular spectral density.

\subsection{Relation to Some Pseudo-random Graphs}
As we have already mentioned in the Introduction, symmetric binary or sign matrices naturally correspond to unweighted, undirected graphs. Let us start by comparing our pseudo-random construction to the truly random graphs. In Figure \ref{fig:rand_vs_prand} we drew a pseudo-random graph on $n=2^5-1=31$ vertices (it was generated using $f_{5}(x) = x^5+x^2+1$ with the initial seed $[1 1 0 1 0]$) versus a truly random graph on the same number of vertices. The vertices are arranged on a circular diagram to emphasize the circulant pattern of the pseudo-random construction which is invariant under the rotations by a multiple of $\frac{2\pi}{31}$ radians shifting the vertices cyclically. It is easy to see that the second graph does not possess such rotational symmetry with the given arrangement of vertices. In fact, any arrangement of vertices in a truly random graph will likely not show any invariance to permutation of vertices, since it can be proven that for $n \to \infty$ a truly random graph has almost surely no nontrivial automorphisms.

It is also instructive to compare the matrices at hand to some known pseudo-random graphs (note that \textit{pseudo-random graphs} is a reserved title for a specific family of graphs; for this reason we do not apply this terminology to our pseudo-random \textit{matrices}). Such graphs are often used to replace truly random graphs in various algorithms as a means of derandomization \cite{alon2004probabilistic}, since they resemble the truly random graphs in some of their properties and are easy to generate and access. 

One of the standard and probably most well studied examples of pseudo-random graphs is the Paley graph. Let $q$ be a prime power which is congruent to $1$ modulo $4$ so that $-1$ is a square in the finite field $GF(q)$. The Paley graph is build in the following way. Its vertices are all elements of $GF(q)$ and two vertices are adjacent if their difference is a quadratic residue in this field \cite{brouwer2011spectra, krivelevich2006pseudo}. As Figure \ref{fig:paley} demonstrates, the empirical spectra of the adjacency matrices of the Paley graph (lifted by $\zeta$ defined in (\ref{eq:zeta_def})) consist of three spikes (see Proposition 9.1.1 from \cite{brouwer2011spectra} for the proof). This in particular shows that the spectra of Paley graphs can not approach any smooth limiting law let alone semicircle. 
\begin{figure}
\centering
\includegraphics[width=3.75in]{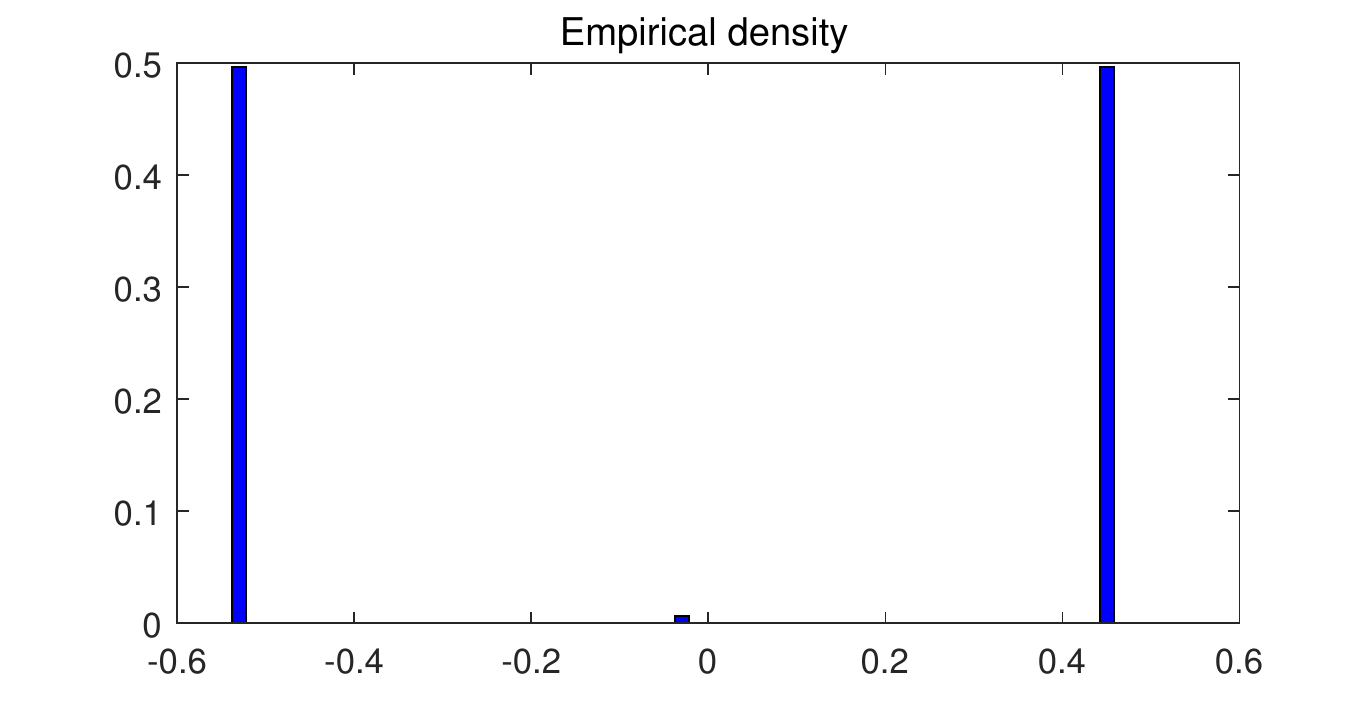}
\caption{Empirical spectral density of the Paley graph, $n=157$.}
\label{fig:paley}
\end{figure} 

\begin{figure*}[!t]
\hspace{-2.3cm}
\includegraphics[width=8.6in]{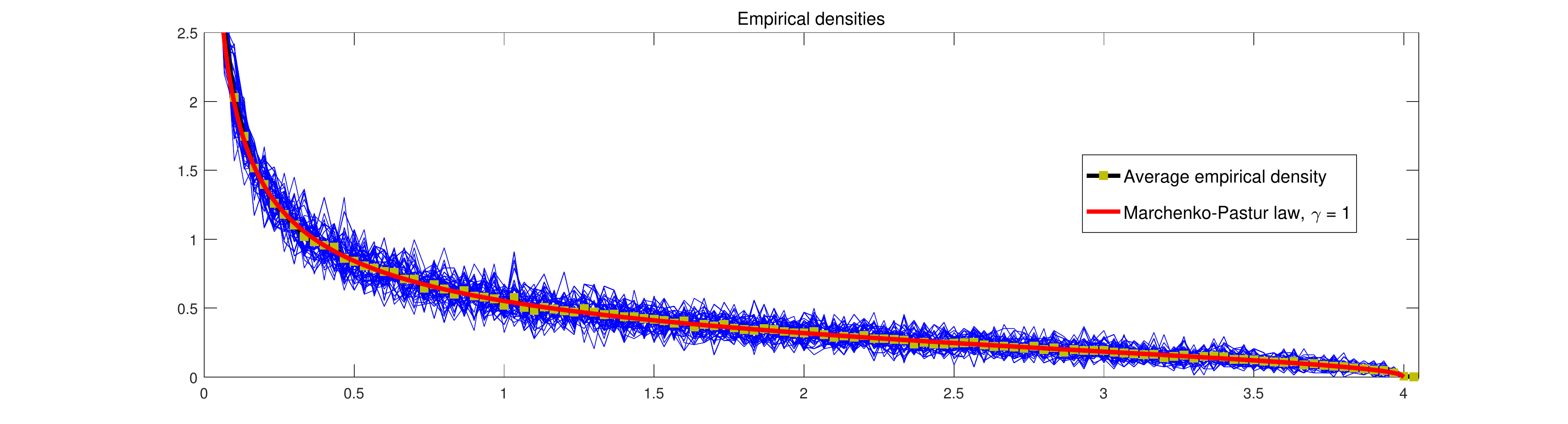}
\caption{Empirical spectral densities of $100$ matrices $\A_n^2(a)$, $m=13\;\;\; (n=8191)$.}
\label{fig:mp_law}
\end{figure*} 

\subsection{The Marchenko-Pastur Law}
As proven in Section \ref{sec:mp_subseq}, the empirical spectra of squares of the symmetric pseudo-random matrices at hand converge to the Marchenko-Pastur law with $\gamma=1$. Figure \ref{fig:mp_law} compares the empirical distributions of eigenvalues of $4\A_n^2$ and their average to the Marchenko-Pastur distribution with the same aspect ratio (\ref{eq:mp_law}).

\bigskip
All the phenomena mentioned in this section require further investigation and will hopefully lead to better understanding of the \emph{inverse spectral problem} for sign matrices.

\section{Conclusions}
\label{sec:conc}
In this article, we consider the problem of generating pseudo-random matrices based on the similarity of their spectra to Wigner's semicircular law. Using binary $m$-sequences, we give a simple explicit construction of a family of $n \times n$ circulant matrices whose spectra converge to the semicircular law. We show that the Kolmogorov complexity of the suggested construction is as low as $2\log_2(n)$. We compare our pseudo-random matrices with other random and pseudo-random matrix models and outline the directions of our current and future research. In addition, we propose pseudo-random matrices of the same algorithmic complexity with spectra converging to the Marchenko-Pastur law with the aspect ratio one and currently work on extending this construction to general aspect ratios.

\appendix
The goal of this section is to prove Propositions \ref{lem:main_exp} and \ref{lem:main_var}. We start from auxiliary notation and claims.

\subsection{Spectra of Subcodes of Algebraic-geometric Codes}
\label{app:spectr_codes}
In this section we formulate a technical result from the theory of algebraic-geometric codes that we use in the proof later. Due to the lack of space we cannot provide an extensive introduction into the topic and refer the interested reader to the paper \cite{vladuts1991spectra} and the book \cite{tsfasman2013algebraic}.

\begin{definition}
\label{def:weight_func}
The weight of a binary $(0/1)$ sequence is the number of ones in it. Given a binary linear code $\mathcal{B}$, denote by $\pi_{\mathcal{B}}(l)$ its spectral function, that is the number of codewords from $\mathcal{B}$ with weight $l$.
\end{definition}

Let $\mathcal{X}$ be a smooth projective absolutely irreducible algebraic curve of genus $g$ over $GF(q),\; q=p^h,\; p$ is prime. $\mathcal{X}(GF(q))$ denotes the set of $GF(q)$-points on it. By $K$ we denote the algebraic closure of $GF(q)$. $GF(q)(\mathcal{X})$ denotes the field of rational curves on $\mathcal{X}$, defined over $GF(q)$. Let $G = \sum_{i=1}^s a_iP_i$ be a divisor over $\mathcal{X}$, where $P_1,\dots,P_s$ are points on $\mathcal{X}$ defined in general over $K$. We shall assume that the divisor itself is defined over $GF(q)$, i.e. it is mapped into itself by a transformation raising the coordinates of all the points into the power of $q$. The set $\suppp{P_1,\dots,P_s}$ is referred to as the support of $G$ and we shall denote its cardinality by $s(G)$. Let $\{Q_1,\dots,Q_n\} = \mathcal{X}(GF(q))\backslash \suppp{G}$ and denote by $t=|\mathcal{X}(GF(q)) \cap \suppp{G}|$ the number of $GF(q)$-points in the support of $G$. For any function $f \in k(\mathcal{X})^*$, assemble the divisor of poles $(f)_{\infty}$ and the divisor of zeros $(f)_0$. These are effective divisors equal to the sum of $K$-points $\mathcal{X}$ where $f$ has a pole or a zero with their respective multiplicities. The divisor of $f$ is $(f) = (f)_0 - (f)_\infty$. Given $G$, build the space of functions $L(G) = \{f \in GF(q)(\mathcal{X})^* | (f) + G \geqslant 0\} \cup \{0\}$. Every divisor can be decomposed as $G = G_{+}-G_{-}$, where $G_{+}$ and $G_{-}$ are both effective divisors with disjoint supports. Denote $D=\sum_{i=1}^n Q_i$, then the algebraic-geometric code $\Gamma(D,G)$ is defined as the $q$-ary code of length $n$ with the parity-check matrix $\norm{f_i(Q_j)}$, where $f_i$ runs over the basis of $L(G)$. We also introduce the $p$-ary code $S(D,G) = \Gamma(D,G) \cap GF^n(p)$ which is a subcode of $\Gamma(D,G)$ over the subfield $GF(p)$. 

Let $G=G_{+}-P,\; P \in \mathcal{X}(GF(q))$. If $G_{+} = \sum_{i=1}^m b_iP_i,\; b_i > 0,\; i=1,\dots,m$, then set $G_i = \sum_{i=1}^m\[b_i/p\]P_i - P$, where $[x]$ is the integer part of $x\in \mathbb{R}$ and denote $\alpha=\dim{L(G)} - \dim{L(G_1)}$.

\begin{lemma}[Theorem 7.2 from \cite{vladuts1991spectra}]
\label{lemma:code_lengths_orig}
Let $p=2,\; G=G_{+}-P,\; n \geqslant \max[11, 2 \deg(G)]$. The spectrum of the code $\mathcal{B} = S(D,G)$ satisfies
\begin{multline}
\left|\pi_{\mathcal{B}}(l) - {n \choose l}2^{-h\alpha}\right| \leqslant \frac{n^{l/2}}{2}\exp\big(2n^{-1/2}(t(G) \\+(2g-2+s(G_{+})+\deg(G_{+}))2^{h/2})\big),
\end{multline}
where $1 \leqslant l \leqslant \[\frac{n}{2}\]$. If $n > 2((2g-2+s(G_{+})+\deg(G_{+}))2^{h/2}+t(G))$, then $n - \dim{S(D,G)} = h\alpha$.
\end{lemma}

Bellow we shall make use of the following version of this lemma adapted to our setting. Denote the dimension of the code by $k$.

\begin{lemma}[Corollary of Theorem 7.2 from \cite{vladuts1991spectra}]
\label{lemma:code_lengths}
In the setting of Lemma \ref{lemma:code_lengths_orig}, for any fixed $l \in \N$ there exists a constant $c(l)$ not depending on $k$ or $n$, such that
\begin{equation}
\left|\pi_{\mathcal{B}}(l) - {n \choose l}2^{-(n-k)}\right| \leqslant c(l) n^{l/2},
\end{equation}
for all $n$ large enough.
\end{lemma}


\subsection{Auxiliary Results on Binary Polynomials}
\label{app:aux_res}
Given a binary polynomial $f(x)$, its reciprocal is a polynomial of the same degree defined as
\begin{equation}
\label{eq:recipr_def}
\hat{f}(x) = x^{\deg f} f\(x^{-1}\).
\end{equation}
A polynomial is called \textit{self-reciprocal} if it is equal to its reciprocal. It is common to identify the codewords of linear codes with polynomials. More details on that can be found in standard textbooks on coding theory, e.g. \cite{macwilliams1977theory}. Below we often use codewords and polynomials interchangeably and mention that if the degree of the polynomial is less than $n-1$, where $n$ is the code length, then when such polynomial is turned into a codeword, its missing high degree monomials are considered as having zero coefficients and the corresponding codeword is padded by zeros in those locations.

\begin{lemma}[Problem 6.7.c from \cite{peterson1972error}]
\label{lemma:prim_selfrecipr}
The only self-reciprocal binary primitive polynomials are
\begin{equation}
x+1,\qquad\text{and}\qquad x^2+x+1.
\end{equation}
\end{lemma}

Given a binary codeword $\c$ of length $n$, its \textit{inverse} codeword is the codeword of the same length with the order of the elements inverted,
\begin{equation}
\label{eq:rev_def_tilde}
\tilde{\c}(x) = x^{n-1}\c\(x^{-1}\).
\end{equation}
We say that a codeword is a \textit{palindrome} if it coincides with its inverse. Note that in the definition of self-reciprocal polynomials, $f\(x^{-1}\)$ is multiplied by $x$ raised in the power of the polynomial's degree, see equation (\ref{eq:recipr_def}), while here we care about the length of the codeword. Therefore, self-reciprocal polynomials may correspond to non-palindromic codewords (e.g. if the code length is larger than the degree) and vise versa.

As an immediate corollary of Lemma \ref{lemma:prim_selfrecipr} we get
\begin{corollary}
\label{cor:palindr_subsp}
Let $\mathcal{B}$ be a Hamming code of order $n\geqslant 4$ generated by the primitive polynomial $f(x)$ with $m=\deg{f}$, then the linear subcode $\mathcal{D} \subset \mathcal{B}$ consisting of all palindromes in $\mathcal{B}$ has dimension
\begin{equation}
\dim{\mathcal{D}} = \frac{n+1}{2}-m.
\end{equation}
\end{corollary}
\begin{proof}
Since $f(x)$ is a primitive polynomial and $n\geqslant 4$, the degree of $f(x)$ is at least $m = \log_2 (n+1) > 2$, therefore, $f(x)$ is not self-reciprocal by Lemma \ref{lemma:prim_selfrecipr}. Let $\c = z(x)f(x) \in \mathcal{B}$ be a palindrome, then
\begin{equation}
\label{eq:pal_c_eq}
\c = x^{n-1}\c\(x^{-1}\),
\end{equation}
and
\begin{equation}
z(x)f(x) = x^{n-1}z\(x^{-1}\)f\(x^{-1}\).
\end{equation}
Since $f(x)$ is primitive and in particular irreducible, this implies that $\c$ must be of the form
\begin{equation}
\label{eq:D_isom}
\c(x) = f(x)\hat{f}(x)g(x).
\end{equation}
However, from (\ref{eq:pal_c_eq}) and (\ref{eq:recipr_def}) we conclude that
\begin{multline}
f(x)\hat{f}(x)g(x) = x^{n-1}f\(x^{-1}\)\hat{f}\(x^{-1}\)g\(x^{-1}\) \\ = x^{n-2m-1}f(x)\hat{f}(x)g\(x^{-1}\),
\end{multline}
and thus $g(x)$ must be a palindrome of length $n-\deg{f(x)}-\deg{\hat{f}(x)} = n-2m$,
\begin{equation}
g(x) = x^{n-2m-1}g\(x^{-1}\).
\end{equation}
The degree of $g(x)$ is at most its length minus one,
\begin{equation}
\deg{g(x)} \leqslant n-2m-1.
\end{equation}
which is always an odd number. Finally, all palindromes of degrees up to $n-2m-1$ with the addition $\rm{mod}$ $\frac{x^n-1}{f(x)\hat{f}(x)}$ form a linear space of dimension $\frac{n+1}{2}-m$ which is isomorphic to $\mathcal{D}$ through (\ref{eq:D_isom}). This completes the proof.
\end{proof}

\subsection{Proofs of the Main Results}
\label{app:proof_main_res}
In this section we prove Propositions \ref{lem:main_exp} and \ref{lem:main_var} that guarantee the almost sure weak convergence of the spectra of our pseudo-random matrices to the semicircular law. Both proofs follow the same ideas as the original proofs of the two equations in Lemma \ref{lem:mom_conv_norm} by Wigner \cite{wigner1955characteristic} and Arnold \cite{arnold1967asymptotic}, respectively. Since these ideas will be used below, for the reader's convenience we start by sketching the proof of Lemma \ref{lem:mom_conv_norm}.

\begin{proof}[Sketch of the Proof of Lemma \ref{lem:mom_conv_norm}, \cite{bai2010spectral, kemp2013math}]
First we show that the expectations of the moments $\beta_r(\W_n)$ of empirical spectral measures of the matrices $\W_n \in \mathcal{W}_n$ (defined in (\ref{eq:mom_def_trace})) converge to the corresponding moments of the semicircular law. Recall that the ensemble of Wigner's matrices $\mathcal{W}_n$ is symmetric, meaning that for any $\W_n$ in it, $-\W_n$ also belongs to $\mathcal{W}_n$, therefore, the expectations of odd moments of $\W_n$ are zero. 

Formula (\ref{eq:mom_def_trace}) suggests that we can instead investigate the behavior of the traces $\frac{1}{n}\Tr{\W_n^r}$. Denote the elements of $\W_n = \{w_{ij}\}_{i,j=0}^{n-1}$ and consider the expansion
\begin{multline}
\label{eq:trace_def}
\beta_r(\W_n) = \frac{1}{n}\Tr{\W_n^r} \\ = \frac{1}{n}\sum_{i_0,\dots,i_{r-1}=0}^{n-1}w_{i_0i_1}w_{i_1i_2}\cdot\cdot\cdot w_{i_{r-2}i_{r-1}}w_{i_{r-1}i_0}.
\end{multline}
Let $K$ be a fully connected graph on $n$ vertices numbered from $0$ to $n-1$ and denote its edges by the pairs of vertices $(ij)$. A closed path (cycle) of length $r$ in $K$ is an ordered sequence of $r$ edges
\begin{equation}
\i = \{(i_0i_1),(i_1i_2),\dots,(i_{r-2}i_{r-1}),(i_{r-1}i_0)\},
\end{equation}
such that the first vertex of the first edge coincides with the second vertex of the last edge and the first vertex of any edge coincides with the second vertex of the previous edge. 
Given these definitions, formula (\ref{eq:trace_def}) receives the following interpretation: $\Tr{\W_n^r}$ is a sum of products of the matrix elements of the form $w_{i_0i_1}w_{i_1i_2}\cdot\cdot\cdot w_{i_{r-2}i_{r-1}}w_{i_{r-1}i_0}$ over all cycles $\i$ of length $r$ in $K$. 

Next, for even $r$ we calculate the expected value of the scaled moment $\frac{1}{n}\Tr{\W_n^r}$,
\begin{multline}
\label{eq:wign_pr_exp}
\mathbb{E}\[\beta_r(\W_n)\] = \mathbb{E}\[\frac{1}{n}\Tr{\W_n^r}\] \\ = \frac{1}{n}\sum_{i_0,\dots,i_{r-1}=0}^{n-1}\mathbb{E}\, w_{i_0i_1}\cdot\cdot\cdot w_{i_{r-1}i_0}.
\end{multline}
Using the standard convention, in the sequel we call a cycle in which every edge is traversed an even number of times an \textit{even cycle}. Following the classical proof \cite{kemp2013math}, in order to compute (\ref{eq:wign_pr_exp}) we break the right-hand side into three sums
\begin{equation}
\label{eq:3_term_eq}
\mathbb{E}\[\beta_r(\W_n)\] = \text{I} + \text{II} + \text{III},
\end{equation}
where the three summands in the last line correspond to
\begin{enumerate}[label=\Roman*,align=Center]
\item the cycles with less than $\frac{r}{2}$ different vertices,
\item even cycles with exactly $\frac{r}{2}$ different vertices,
\item the cycles with more than $\frac{r}{2}$ different vertices or non-even cycles on $\frac{r}{2}$ different vertices.
\end{enumerate}
The third sum III is the easiest to deal with since it necessarily contains an edge traversed only once, and therefore, due to the independence of matrix elements, the expectation over it is
\begin{equation}
\text{III} = 0.
\end{equation}
As explained in \cite{kemp2013math}, the second sum is the central part of the calculation as it is the leading term of the sum in (\ref{eq:3_term_eq}). It can be easily seen that all the edges in the cycles counted in II have their edges traversed exactly twice, thus, the expectations over them are
\begin{equation}
\mathbb{E}\, w_{i_0i_1}\cdot\cdot\cdot w_{i_{r-1}i_0} = \frac{1}{2^rn^{r/2}},\quad \text{for } \i \text{ - even}.
\end{equation}
Hence, we only need to calculate the number of such cycles. An accurate combinatorial argument \cite{kemp2013math} shows that as $n \to \infty$ the amount of such cycles grows as $C_{r/2}n^{r/2}+O(n^{r/2-1})$, and therefore
\begin{equation}
\text{II} = \frac{1}{2^r}C_{r/2} + O\(\frac{1}{n}\),
\end{equation}
where $C_r$ is the $r$-th Catalan number as defined in (\ref{eq:catal_num}). It is shown in \cite{kemp2013math}, that when $n$ grows the number of paths in I is asymptotically negligible with respect to that in II. They calculate the exact asymptotics and prove that
\begin{equation}
\text{I} = O\(\frac{1}{n}\).
\end{equation}
Overall, we get
\begin{equation}
\mathbb{E}\[\beta_r(\W_n)\] = \frac{1}{2^r}C_{r/2} + O\(\frac{1}{n}\).
\end{equation}

Next we need to show that the variance of this estimate is summable (\ref{eq:var_cond}) for every $r$. Consider the expression
\begin{align}
\label{eq:var_calc_demonst}
&\var{\beta_r(\W_n)} = \var{\frac{1}{n}\Tr{\W_n^r}} \nonumber \\ 
&\qquad= \frac{1}{n^2}\(\mathbb{E}\[\Tr{\W_n^r}\]^2 - \[\mathbb{E}\,\Tr{\W_n^r}\]^2\) \nonumber \\
&\qquad= \frac{1}{n^2}\sum_{\i,\j} \mathbb{E}\[w_\i w_\j\] - \mathbb{E}\[w_\i\] \mathbb{E}\[w_\j\],
\end{align}
where
\begin{equation}
\i = \{(i_0i_1),(i_1i_2),\dots,(i_{r-2}i_{r-1}),(i_{r-1}i_0)\},
\end{equation}
\begin{equation}
w_\i = w_{i_0i_1}w_{i_1i_2}\cdot\cdot\cdot w_{i_{r-2}i_{r-1}}w_{i_{r-1}i_0},
\end{equation}
and analogously for $\j$. If the cycles $\i$ and $\j$ have no common edges, then $w_\i$ and $w_\j$ are independent and the difference in (\ref{eq:var_calc_demonst}) is zero. If the cycles $\i$ and $\j$ have common edges, then their union $\i \cup \j$ (after certain cyclic shifts of edges in $\i$ and $\j$, if necessary) is also a cycle \cite{kemp2013math}. Therefore, in what follows we only consider cycles $\k = \i \cup \j$ obtained by gluing $\i$ and $\j$ together. 

Here again, the main contribution to the sum is made by the even $\k$-s for which each edge is traversed exactly twice. As can be seen from (\ref{eq:var_calc_demonst}), the contribution of every such cycle to the variance is at most $(2^{2r}n^{n+2})^{-1}$. The remaining task is to count the number of such cycles. Below we do not utilize the exact combinatorial argument usually used to get the desired amount of cycles and due to lack of space omit it, however, we refer the reader to the lecture notes \cite{kemp2013math} explaining the calculation in detail. They show that the total number of such cycles is $O(n^r)$. Thus, overall we get
\begin{equation}
\var{\beta_r(\W_n)} = O\(\frac{n^r}{n^2n^r}\) = O\(\frac{1}{n^2}\),
\end{equation}
which completes the proof.
\end{proof}

\begin{rem}
\label{rem:better_conv_var}
In fact, a tight lower asymptotic bound can also be established. As can be seen from the argument at the end of the proof in \cite{kemp2013math}, the number of even cycles $\k$ in which every edge is traversed twice and either $\i$ or $\j$ (or both) is not even is proportional to $n^r$. On each such $\k$, either $\mathbb{E}\[w_\i\] = 0$ or $\mathbb{E}\[w_\j\]=0$, therefore, the contribution of this cycle to the sum in (\ref{eq:var_calc_demonst}) is exactly $(2^{2r}n^{n+2})^{-1}$. Consequently, the partial sum in (\ref{eq:var_calc_demonst}) corresponding to these cycles is bounded from below by $\Theta\(\frac{1}{n^2}\),\; n \to \infty$, and overall
\begin{equation}
\label{eq:better_conv_var}
\var{\beta_r(\W_n)} = \Theta\(\frac{1}{n^2}\),
\end{equation}
where $f(n)=\Theta(g(n))$ means that $f(n)=O(g(n))$ and $g(n)=O(f(n)),\; n \to \infty$.
\end{rem}

\begin{proof}[Proof of Proposition \ref{lem:main_exp}]
The proof utilizes the method of moments and follows similar steps as the sketch of Lemma \ref{lem:mom_conv_norm}. The idea is to decompose the expected moments into three sums similarly to equation (\ref{eq:3_term_eq}). As we show below, the first two sums can be bounded analogously, however, the treatment of sum III is different. Unlike Wigner's case where III is zero, in our case we can only guarantee
\begin{equation}
\text{III} = O\(\frac{\text{II}}{n}\),
\end{equation}
nevertheless, as can be seen from the proof, this is enough to get the desired claim. Next we provide all the details.

\setcounter{equation}{68}
\begin{figure*}[bp]
\begin{equation}
\label{eq:t_pic_eq}
\qquad \begin{matrix}
\hline \\ \qquad
\begin{matrix}\nu \colon \;\; (t_0,\dots,t_{r-1}) \mapsto \\ \\ \\ \end{matrix} \quad \begin{tabular}{*{11}{c}}
  \hline
  \multicolumn{1}{|c|}{0} & \multicolumn{1}{c|}{1} & \multicolumn{1}{c|}{0} & \multicolumn{1}{c|}{1} & \multicolumn{1}{c|}{$\dots$} & \multicolumn{0}{c|}{0} & \multicolumn{0}{c|}{0} & \multicolumn{1}{c|}{$\dots$} & \multicolumn{1}{c|}{1} & \multicolumn{1}{c|}{0} & \multicolumn{1}{c|}{1} \\ \hline
  $\uparrow$ & $\uparrow$ & $\uparrow$ & $\uparrow$ & & $\uparrow$ & $\uparrow$ & & $\uparrow$ & $\uparrow$ & $\uparrow$ \\
  $0$ & $t_{q_2}$ & $t_{q_1}, t_{q_4}$ & $-t_{q_3}$ & & $\frac{n-1}{2}$ & $\frac{n+1}{2}$ & & $t_{q_3}$ & $-t_{q_1}, -t_{q_4}$ & $-t_{q_2}$
\end{tabular} \qquad
\end{matrix}
\end{equation}
\end{figure*}
\setcounter{equation}{64}

Here again, the ensemble $\mathcal{A}_n$ is symmetric (with every matrix $\A_n$ it contains its negative $-\A_n$), thus all the odd expected moments are zeros and we only focus on the even ones. Using definition (\ref{eq:def_a_eq}) of the matrix elements $a_{ij}$, similarly to formula (\ref{eq:wign_pr_exp}), the $r$-th moment of $\A_n$ reads as
\begin{align}
\label{eq:trace_av}
\mathbb{E}&\[\beta_r(\A_n)\] = \mathbb{E}\[ \frac{1}{n}\Tr{\A_n^r(a)}\] \\ 
&\;\;= \frac{1}{n^2}\sum_{a=0}^{n-1} \frac{1}{2^rn^{r/2}} \nonumber \\
&\hspace{0.7cm} \times \sum_{i_0,\dots,i_{r-1}=0}^{n-1} (-1)^{\sum_{q=0}^{r-1} \varphi(i_{q+1}-i_q+a) + \varphi(i_q - i_{q+1}+a)}, \nonumber
\end{align}
where we treat the indices $q$ of the vertices $i_q$ modulo $r$. Let
\begin{equation}
\label{eq:t_def_i}
t_q = i_{q+1}-i_q \mod n,\quad q=0,\dots,r-1.
\end{equation}
We denote the $r$-tuple $(t_0,\dots,t_{r-1})$ by
\begin{equation}
\t = (t_0,\dots,t_{r-1}) \in [n]^r.
\end{equation}

Define a function
\begin{align}
\label{eq:eq_40}
&\qquad\qquad\quad \nu : [n]^r  \to GF(2)^n, \nonumber\\
(t_0,&\dots,t_{r-1}) \\
& \mapsto \bigg\{\sum_{q=0}^{r-1} \mathbbm{1}(t_q=i)+\mathbbm{1}(-t_q=i) \;\mmod 2 \bigg\}_{i=0}^{n-1}, \nonumber
\end{align}
where $\mathbbm{1}$ is an indicator function and the equalities are modulo $n$. Equation (\ref{eq:t_pic_eq}) schematically illustrates the action of $\nu(\cdot)$. In words, $\nu(\cdot)$ does the following. It takes the $r$-tuple $\t = (t_0,\dots,t_{r-1})$ and first maps it into an extended $2r$-tuple $\(\t,-\t\) = \(t_0,\dots,t_{r-1},-t_0,\dots,-t_{r-1}\) \in [n]^{2r}$. Then it calculates the number of appearances of every number $t \in [n]$ in this $2r$-tuple, which we denote by $\#\{t\}$ and constructs a codeword $\c \in GF(2)^n$ by setting its elements with indices $t$ to $\#\{t\} \;\mmod 2$ and zeros otherwise. 

Consider an example. Let $n=7$ and $\t=(3,5,0)$, then $r=3$ and the $2r$-tuple $\(\t,-\t\) = \{3,5,0,-3,-5,0\} = \{3,5,0,4,2,0\}$. In this case $0$ appears twice, $2,3,4$ and $5$ appear only once and the remaining numbers $1,6$ appear zero times. Therefore, we get $\nu(\t) = (0,0,1,1,1,1,0)$, where ones appear on the $2$d, $3$d, $4$th and $5$th places and zeros otherwise. Also note that the bit with index $0$ (the leftmost bit) of $\nu(\t)$ is always zero since the number of zeros in the $2r$-tuple $\(\t,-\t\)$ must always be even.

Recall that those cycles in which every edge is traversed even number of times we call even. Using the relations (\ref{eq:t_def_i}), it can now be easily shown that if the original cycle $\{(i_0i_1),\dots,(i_{r-1}i_0)\}$ was even, then the number of appearances of every number from $[n]$ in the $2r$-tuple $\(t_0,\dots,t_{r-1},-t_0,\dots,-t_{r-1}\)$ is even, and therefore $\nu(\t) = \bm{0}$ for that $\t$.

\setcounter{equation}{69}

Due to (\ref{eq:t_def_i}), we restrict the $r$-tuples $\t$ under consideration to satisfy the following linear relation
\begin{equation}
\label{eq:lin_cond}
\sum_{q=0}^{r-1} t_q = 0.
\end{equation}
Denote the set of legitimate $r$-tuples $\t$ by
\begin{equation}
\mathcal{T}_r = \Big\{\t \in [n]^r \Big| \sum_{q=0}^{r-1} t_q =0 \Big\}.
\end{equation}
To simplify the treatment of the last expression in (\ref{eq:trace_av}), define a function 
\begin{align}
\label{eq:tau_def}
&\tau(\nu(\t);a) \\
&\quad= \begin{cases}\sum_{q=0}^{r-1} \big[ \varphi(t_q+a) + \varphi(-t_q+a)\big] & \mmod 2, \\ &\nu(\t) \neq \bm{0} \\ \qquad\qquad\qquad\quad 0, &\nu(\t) = \bm{0}.\end{cases} \nonumber
\end{align}
As a consequence of the Shift-and-add Property [P\ref{prop:s-a-a}] of the Golomb sequences discussed in Section \ref{sec:gol_a}, for any fixed $\t,\; \tau(\nu(\t);a)$ viewed as a function of $a$ is a shift of the original Golomb sequence or a zero sequence. If $\tau(\nu(\t);a)$ is a valid Golomb sequence, then it follows from Axiom [G\ref{ax3}] that the sum of the powers $(-1)^{\tau(\nu(\t);a)}$ over $a$ is $-1$. On the other hand, if $\tau(\nu(\t);a)$ is a zero function of $a$, then this sum equals $n$,
\begin{equation}
\label{eq:autocor}
\sum_{a=0}^{n-1} (-1)^{\tau(\nu(\t);a)} = \begin{cases} -1, &\tau(\nu(\t);a) \not\equiv 0, \\ n, &\tau(\nu(\t);a) \equiv 0. \end{cases}
\end{equation}

Following the classical method of moments explained in the sketch of the proof of Lemma \ref{lem:mom_conv_norm} above, rewrite expression (\ref{eq:trace_av}) as
\begin{align}
\label{eq:trace_av1}
2^rn^{r/2+1}\; \mathbb{E}& \[\beta_r(\A_n)\] = \frac{1}{n}\sum_{a=0}^{n-1}\, \sum_{\t \in \mathcal{T}_r} (-1)^{\tau(\nu(\t);a)} \nonumber \\ 
& \quad\;\; = \text{I} + \text{II} + \text{III}, 
\end{align}
where we have multiplied both sides of (\ref{eq:trace_av}) by $2^rn^{r/2+1}$ to make  the notation less bulky. The three sums in the last line of (\ref{eq:trace_av1}) correspond to
\begin{enumerate}[label=\Roman*,align=Center]
\item the cycles with less than $\frac{r}{2}$ different vertices,
\item even cycles with exactly $\frac{r}{2}$ different vertices,
\item the cycles with more than $\frac{r}{2}$ different vertices or non-even cycles on $\frac{r}{2}$ different vertices.
\end{enumerate}
Exactly the same reasoning as in the sketch above (modulo scaling by $2^rn^{r/2+1}$) leads to the bound
\begin{equation}
\label{eq:I_expr}
\text{I} = O\(n^{r/2}\).
\end{equation}

For the second term II in the sum, as we have already mentioned earlier, $\nu(\t) = \bm{0}$ and hence $\tau(\nu(\t);a)$ is a zero function of $a$. Therefore, the expectation over every even cycle gives the same contribution $1$ to the sum no matter what are the dependencies of the edge weights along such cycle, and we only need to calculate the total number of these cycles exactly as in Wigner's case explained in the sketch. We conclude that
\begin{equation}
\label{eq:II_expr}
\text{II} = C_{r/2}n^{r/2+1} + O\(n^{r/2}\).
\end{equation}

The only remaining sum that needs to be further investigated is III. In fact, below we show that the contribution of III is $O\(n^{r/2}\)$ which is enough to get the desired statement.

Since III does not contain summands over even paths, below we assume that $\nu(\t) \neq \bm{0}$. In order to calculate the value of III, we need to determine the number $\Gamma_{0}$ of $r$-tuples $\t \in \mathcal{T}_r$, for which $\tau(\nu(\t);a)$ is a zero function and the number $\Gamma_{g}$ of $r$-tuples $\t \in \mathcal{T}_r$, for which $\tau(\nu(\t);a)$ is a valid Golomb sequence (a shift of the original sequence $\varphi$). Using (\ref{eq:autocor}), we may rewrite III as
\begin{equation}
\text{III} = \frac{1}{n}\((-1)\Gamma_{g} + n\Gamma_{0}\).
\end{equation}

Recall that $\mathcal{C}$ stands for the simplex code generated by the Golomb sequence $\varphi$ and $\mathcal{C}^\perp$ denotes its dual code. $\mathcal{C}^\perp$ is a Hamming code generated by a primitive polynomial of degree $m$ and has dimension $k^\perp = n-m$. 

Our next goal is to calculate $\Gamma_{g}$ and $\Gamma_{0}$. The central observation here is that by the definition of the dual code, function $\nu(\cdot)$ maps those $\t \in \mathcal{T}_r$ which make $\tau(\nu(\t);a)$ to be a zero sequence into a subset of the dual code $\mathcal{C}^\perp$. Denote this set by
\begin{equation}
\mathcal{H} = \Big\{\nu(\t) \in \mathcal{C}^\perp \Big| \t \in \mathcal{T}_r  \Big\},
\end{equation}
The weights of the codewords of $\mathcal{H}$ obtained from our $r$-tuples $\t$ run through the range $2r,\,2r-4,\,2r-8,\,\dots,\,0$ (the last element is necessarily $0$ since $r$ is even and thus $4$ divides $2r$). The set of the codewords of the form $\nu(\t),\; \t \in \mathcal{T}_{\tau}$ of the same length $n$ is denoted by
\begin{equation}
\mathcal{P} = \Big\{\nu(\t) \in GF(2)^n \Big| \t \in \mathcal{T}_r \Big\}.
\end{equation}
To simplify the explanation of the argument we provide next, let us forget about the condition (\ref{eq:lin_cond}) for a moment. Consider two codes defined as
\begin{align}
\label{eq:eq_52}
\mathcal{H}' &= \Big\{\nu(\t) \in \mathcal{C}^\perp \Big\}, \\
\mathcal{P}' &= \Big\{\nu(\t) \in GF(2)^n \Big\}. \nonumber
\end{align}
Later we will explain how to return (\ref{eq:lin_cond}) back.

Analogously to the above definitions, let $\Gamma_{0}'$ be the number of $r$-tuples $\t \in \[n\]^r$, for which $\tau(\nu(\t);a)$ is a zero function and $\Gamma_{g}'$ be the number of $r$-tuples $\t \in \[n\]^r$, for which $\tau(\nu(\t);a)$ is a valid Golomb sequence. Now our goal is to calculate $\Gamma_{0}'$ and $\Gamma_{g}'$ so that we will be able to bound the value of
\begin{multline}
\label{eq:III_1}
\text{III}' = \frac{1}{n} \sum_{a=0}^{n-1} \sum_{\nu(\t) \neq \bm{0}} (-1)^{\tau(\nu(\t);a)} =  \sum_{\nu(\t) \neq \bm{0}} \mathbb{E} (-1)^{\tau(\nu(\t);a)} \\ = \frac{1}{n}\((-1)\Gamma_{g}' + n\Gamma_{0}'\). 
\end{multline}
Since III$'$ involves only such $r$-tuples $\t$ for which $\nu(\t) \neq \bm{0}$, we get
\begin{align}
\Gamma_{0}' &= \sum_{l=1}^{r-1} \rho_r(l)\pi_{\mathcal{H}'}(2l), \label{eq:gam_0}\\
\Gamma_{g}' &= \sum_{l=1}^{r-1} \rho_r(l)\[\pi_{\mathcal{P}'}(2l)-\pi_{\mathcal{H}'}(2l)\], \label{eq:gam_g}
\end{align}
where $\rho_r(l)$ is the number of different $r$-tuples $\t \in \[n\]^r$ producing the same codeword of weight $2l$ and $\pi_{\mathcal{B}}(l)$ is the weight function as in Definition \ref{def:weight_func} above. Note that
\begin{equation}
\rho_r(l) \leqslant (2r)!!,
\end{equation}
and its values is therefore bounded uniformly with respect to $n$. Plugging (\ref{eq:gam_0}) and (\ref{eq:gam_g}) into (\ref{eq:III_1}) gives
\begin{align}
\label{eq:III_12}
\text{III}' & = \frac{1}{n} \sum_{l=1}^r \rho_r(l)\[(-1)(\pi_{\mathcal{P}'}(2l) - \pi_{\mathcal{H}'}(2l)) + n\pi_{\mathcal{H}'}(2l)\] \nonumber \\
& = \frac{1}{n} \sum_{l=1}^r \rho_r(l)\[(n+1)\pi_{\mathcal{H}'}(2l) - \pi_{\mathcal{P}'}(2l)\]. 
\end{align}

In order to compute III$'$, next we need to calculate the spectral functions $\pi_{\mathcal{P}'}(l)$ and $\pi_{\mathcal{H}'}(l)$. Dealing with the codes $\mathcal{P}'$ and $\mathcal{H}'$ themselves is a complicated task and all the known general bounds on the spectral functions are not tight enough for our purposes. Therefore, we suggest the following technical trick. 

As explained after equation (\ref{eq:eq_40}), the zeroth bit of the codeword $\nu(\t)$ is zero for any $\t \in [n]^r$,
\begin{equation}
\c \in \mathcal{P}' \quad \Rightarrow \quad c_{0} = 0.
\end{equation}
In addition, by the definition of $\nu(\cdot)$ the codewords $\nu(\t)$ with the zeroth bit removed are palindromes of length $n-1$. Since $\mathcal{H}' \subset \mathcal{P}'$, both these properties hold for the codewords in $\mathcal{H}'$, as well. Due to these observations, instead of counting the codewords with redundancy, let us drop the zeroth bit from all of them and consider the first halves of the remaining codewords. By construction, there is a natural one-to-one correspondence between the original codes and their shortened versions. Therefore, the bounds on the spectral functions of the latter will provide us with the respective bounds for the spectral functions of $\mathcal{P}'$ and $\mathcal{H}'$. More formally, introduce two auxiliary functions:
\begin{align}
\eta \colon GF(2)^v &\times GF(2)^u \to GF(2)^{v+u}, \nonumber\\
((a_0,\dots,a_{v-1}),(b_0,&\dots,b_{u-1})) \\
&\mapsto (a_0,\dots,a_{v-1},b_0,\dots,b_{u-1}) \nonumber,
\end{align}
which concatenates two codewords, and
\begin{align}
\mu \colon GF(2)^v &\to GF(2)^{v+1}, \\
\(a_0,\dots,a_{v-1}\) &\mapsto \(0, a_0,\dots,a_{v-1}\), \nonumber
\end{align}
which pads the codeword with one additional zero on the left. Consider two codes
\begin{align}
\label{eq:code_isom}
\overline{\mathcal{H}'} = \{\c \in GF(2)^{\frac{n-1}{2}} \mid & \;\eta(\mu(\c),\tilde{\c}) \in \mathcal{H}' \}, \nonumber \\
\overline{\mathcal{P}'} = \{\c \in GF(2)^{\frac{n-1}{2}} \mid & \;\eta(\mu(\c),\tilde{\c}) \in \mathcal{P}' \} \nonumber \\
& = GF(2)^{\frac{n-1}{2}},
\end{align}
where we remind that the $\tilde{\cdot}$ operation defined in (\ref{eq:rev_def_tilde}) reverts the order of the elements in a codeword. In particular, (\ref{eq:code_isom}) suggests that the two pairs of codes $\overline{\mathcal{H}'}$ and $\mathcal{H}'$, and $\overline{\mathcal{P}'}$ and $\mathcal{P}'$ are isomorphic and we obtain
\begin{align}
\label{eq:equiv_no_bar}
\pi_{\mathcal{H}'}(2l) &= \pi_{\overline{\mathcal{H}'}}(l), \nonumber \\
\pi_{\mathcal{P}'}(2l) &= \pi_{\overline{\mathcal{P}'}}(l).
\end{align}

At this point, in order to utilize the results from Appendix \ref{app:spectr_codes} regarding the spectra of linear codes, let us embed the sets $\overline{\mathcal{H}'}$ and $\overline{\mathcal{P'}}$ into linear codes. For this purpose, consider their linear spans $\langle \overline{\mathcal{H}'} \rangle$ and $\langle \overline{\mathcal{P}'} \rangle$. The crucial observation here is that the spectral functions for $l \leqslant r$ do not change, namely, $\pi_{\langle \overline{\mathcal{H}'} \rangle}(l) = \pi_{\overline{\mathcal{H}'}}(l)$ and $\pi_{\langle \overline{\mathcal{P}'} \rangle}(l) = \pi_{\overline{\mathcal{P}'}}(l)$ for $l \leqslant r$. This is clear for $\overline{\mathcal{P}'}$, let us explain that for $\overline{\mathcal{H}'}$ as well. Indeed, if we take the original code $\mathcal{H}'$ and embed it into its linear span $\langle \mathcal{H}' \rangle$, the spectral functions will not change for codewords with weights not exceeding $2r$, since $\mathcal{H}'$ already contains all possible codewods of such weights. Due to definition (\ref{eq:code_isom}), this property directly passes to $\overline{\mathcal{H}'}$ and its span.

From Corollary \ref{cor:palindr_subsp}, we know that the linear subcode of a Hamming code of order $n$ consisting of all palindromes has dimension $\frac{n+1}{2}-m$. However, since $\mathcal{H}'$ was constructed by taking only codewords with the zeroth bit zero, the dimension of its linear span $\langle \mathcal{H}' \rangle$ is decreased by $1$ and becomes 
\begin{equation}
\dim \langle \mathcal{H}' \rangle = \frac{n-1}{2}-m.
\end{equation}
Since $\langle \mathcal{H}' \rangle$ and $\langle \overline{\mathcal{H}'} \rangle$ are isomorphic, we conclude that \begin{equation}
\label{eq:hh_leng_code}
\dim{\langle \overline{\mathcal{H}'} \rangle} = \frac{n-1}{2}-m.
\end{equation}
Recall that $\langle \overline{\mathcal{P}'} \rangle$ contains all possible codewords of length $\frac{n-1}{2}$, thus, its dimension is equal to its length, and totally we get
\begin{equation}
\label{eq:p_leng_code}
\dim{\langle \overline{\mathcal{P}'} \rangle} = \frac{n-1}{2}. 
\end{equation}
As a direct implication of (\ref{eq:p_leng_code}), the spectral function of $\langle \overline{\mathcal{P}'} \rangle$ is given by the binomial coefficients \cite{macwilliams1977theory}
\begin{equation}
\pi_{\langle \overline{\mathcal{P}'} \rangle}(l) = {\frac{n-1}{2} \choose l},
\end{equation}
and consequently,
\begin{equation}
\label{eq:p_weight_est}
\pi_{\overline{\mathcal{P}'}}(l) = {\frac{n-1}{2} \choose l},\quad l \leqslant r.
\end{equation}
Apply Lemma \ref{lemma:code_lengths} to $\langle \overline{\mathcal{H}'} \rangle$ to obtain its spectral function
\begin{equation}
\left|\pi_{\langle \overline{\mathcal{H}'} \rangle}(l) - \frac{1}{n+1}{\frac{n-1}{2} \choose l}\right| = O(n^{l/2}),
\end{equation}
and
\begin{equation}
\label{eq:h_weight_est}
\left|\pi_{\overline{\mathcal{H}'}}(l) - \frac{1}{n+1}{\frac{n-1}{2} \choose l}\right| = O(n^{l/2}),\quad l \leqslant r.
\end{equation}
Using (\ref{eq:equiv_no_bar}), plug the obtained estimates (\ref{eq:p_weight_est}) and (\ref{eq:h_weight_est}) into (\ref{eq:III_12}) to get
\begin{align}
\label{eq:dev_III_tag}
\text{III}' &= \frac{1}{n} \sum_{l=1}^r \rho_r(l)\[(n+1)\pi_{\overline{\mathcal{H}'}}(l) -\pi_{\overline{\mathcal{P}'}}(l)\]  \nonumber \\
&= \frac{1}{n} \sum_{l=1}^r \rho_r(l) \[{\frac{n-1}{2} \choose l}\(\frac{n+1}{n+1} - 1\) + O\(n^{l/2+1}\)\] \nonumber \\
&= \frac{1}{n} O\(n^{r/2+1}\) = O\(n^{r/2}\).
\end{align}

Going one step back, we need to introduce the linear relation (\ref{eq:lin_cond}) to our code. As follows from Lemma \ref{lemma:code_lengths_orig}, in this case $\alpha$ increases by $1$ and this action will reduce the spectral functions of the involved codes by the same multiplicative factor $2^{-m} = \frac{1}{n+1}$ (In the notation of Lemma \ref{lemma:code_lengths_orig} it is $2^{-h\alpha}$) without affecting the tightness of the bound (\ref{eq:dev_III_tag}). Let us plug the results of (\ref{eq:I_expr}), (\ref{eq:II_expr}) and (\ref{eq:dev_III_tag}) into (\ref{eq:trace_av1}) to express the desired expectation,
\begin{align}
\label{eq:trace_av123}
&\mathbb{E}\[\beta_r(\A_n)\] = \frac{\text{I} + \text{II} + \text{III}}{2^rn^{r/2+1}} \nonumber \\
& \quad = \frac{O(n^{r/2}) + C_{r/2}n^{r/2+1} + O\(n^{r/2}\) +  O(n^{r/2})}{2^rn^{r/2+1}} \nonumber \\
& \quad= \frac{C_{r/2}}{2^r} + O\(\frac{1}{n}\),
\end{align}
which completes the proof.
\end{proof}

\begin{proof}[Proof of Proposition \ref{lem:main_var}]
In the calculation of variance we again mimic the proof of \cite{kemp2013math}. All the notation is borrowed from the proof of Proposition \ref{lem:main_exp}. By the definition of variance,
\begin{align}
\label{eq:var_av_3}
&\var{\beta_r(\A_n)} \\ 
&\quad= \frac{1}{n^2}\(\mathbb{E}\[\Tr{\A_n^r(a)}\]^2 - \[\mathbb{E}\,\Tr{\A_n^r(a)}\]^2\) \nonumber\\
&\quad=\frac{1}{2^{2r}n^{r+2}}\[\mathbb{E} \sum_{\t,\,\s \in \mathcal{T}_r} (-1)^{\tau(\nu(\t);a)+\tau(\nu(\s);a)} \right.\nonumber\\ 
&\hspace{3.2cm} \left. - \(\mathbb{E}\sum_{\t \in \mathcal{T}_r} (-1)^{\tau(\nu(\t);a)}\)^2 \]. \nonumber
\end{align}
Similarly to the proof of Proposition \ref{lem:main_exp}, to make the notation less bulky multiply both sides of (\ref{eq:var_av_3}) by $2^{2r}n^{r+2}$ and open the round brackets in the last line to get
\begin{align}
\label{eq:var_av_4}
&2^{2r}n^{r+2}\var{\beta_r(\A_n)} \\ 
&\quad= \sum_{\t,\,\s \in \mathcal{T}_r} \[ \mathbb{E}_a(-1)^{\tau(\nu(\t);a)+\tau(\nu(\s);a)} \right. \nonumber \\ 
&\hspace{3.5cm} - \left. \mathbb{E}_a\mathbb{E}_b(-1)^{\tau(\nu(\t);a)+\tau(\nu(\s);b)}\]. \nonumber
\end{align}
Again, as in the proof of Proposition \ref{lem:main_exp}, for simplicity we drop for a moment conditions (\ref{eq:lin_cond}) on $\t$ and $\s$ and will impose them back later. We break the sum over $\t$ and $\s$ in (\ref{eq:var_av_4}) into two sums,
\begin{align}
\label{eq:var_av_5}
&2^{2r}n^{r+2} \var{\beta_r(\A_n)} \\ 
&\;\;= \(\sum_{\substack{\nu(\t) = \bm{0} \text{ or } \\ \nu(\s) = \bm{0}}} + \sum_{\substack{\nu(\t) \neq \bm{0}, \\ \nu(\s) \neq \bm{0}}}\) \[ \mathbb{E}_a(-1)^{\tau(\nu(\t);a)+\tau(\nu(\s);a)} \right. \nonumber \\ 
&\hspace{1.7cm} - \left. \mathbb{E}_a\mathbb{E}_b(-1)^{\tau(\nu(\t);a)+\tau(\nu(\s);b)}\] = \text{S}_1 + \text{S}_2. \nonumber
\end{align}

Let us first show that
\begin{equation}
\label{eq:s_0}
\text{S}_1 = 0.
\end{equation}
When $\nu(\t) = \bm{0}$ or $\nu(\s) = \bm{0}$, at least one of the summands in the exponents is zero due to the second line in (\ref{eq:tau_def}). Therefore, in this case the expectations in the square brackets are equal and the sum $\text{S}_1$ is zero. 

Denote 
\begin{equation}
\u = \eta(\t,\s) \in [n]^{2r},
\end{equation} 
and, assuming $\nu(\t) \neq \bm{0}$ and $\nu(\s) \neq \bm{0}$, consider the second sum
\begin{align}
\label{eq:var_av_6}
&\text{S}_2 = \sum_{\t,\s} \[ \mathbb{E}_a(-1)^{\tau(\nu(\t);a)+\tau(\nu(\s);a)} \right. \\ 
&\hspace{3cm} - \left. \mathbb{E}_a\mathbb{E}_b(-1)^{\tau(\nu(\t);a)+\tau(\nu(\s);b)}\] \nonumber \\
& = \sum_{\u} \mathbb{E}_a(-1)^{\tau(\nu(\u);a)} - \sum_{\t,\s} \mathbb{E}_a\mathbb{E}_b(-1)^{\tau(\nu(\t);a)+\tau(\nu(\s);b)}. \nonumber
\end{align}
The second line of (\ref{eq:var_av_6}) allows us to get the following expression
\begin{align}
\text{S}_2 &= \(\sum_{\nu(\u) = \bm{0}} + \sum_{\nu(\u) \neq \bm{0}}\) \mathbb{E}_a(-1)^{\tau(\nu(\u);a)} \nonumber \\
&\hspace{3cm} - \sum_{\t,\s} \mathbb{E}_a\mathbb{E}_b(-1)^{\tau(\nu(\t);a)+\tau(\nu(\s);b)} \nonumber \\
&\hspace{2cm} = [\text{U}_0 + \text{U}_g] - \text{TS}, \label{eq:last_l_68}
\end{align}
where by a little abuse of notation we treat $\u \to \nu(\u)$ as a function sending $2r$-tuples (instead of $r$-tuples as before) into binary codewords of length $n$, but otherwise defined exactly as in the proof of Proposition \ref{lem:main_exp}.

The sum $\text{U}_0$ in the last line of (\ref{eq:last_l_68}) corresponds exactly to the even paths, over which correlations of the matrix entries do not matter (as explained in the course of calculation of II in the proof of Proposition \ref{lem:main_exp}). Hence, the reasoning from \cite{kemp2013math} applies verbatim and we get
\begin{equation}
\label{eq:u_0}
\text{U}_0 = O\(n^r\).
\end{equation}

Next, consider the expression 
\begin{align}
\text{U}_g &= \sum_{\nu(\u) \neq \bm{0}} \mathbb{E}_a(-1)^{\tau(\nu(\u);a)} \\
&\hspace{2cm} = \frac{1}{n}\sum_{a=0}^{n-1}\, \sum_{\nu(\u) \neq \bm{0}} (-1)^{\tau(\nu(\t);a)}, \nonumber
\end{align}
and note that it is exactly equal to III$'$ defined in the first line of (\ref{eq:III_1}) with $r$ replaced by $2r$. Hence, we conclude from (\ref{eq:dev_III_tag}) that
\begin{equation}
\label{eq:u_g}
\text{U}_g = O(n^r).
\end{equation}

Let us now focus on the last summand from (\ref{eq:last_l_68}),
\begin{equation}
\text{TS} = \sum_{\nu(\t) \neq \bm{0},\: \nu(\s) \neq \bm{0}} \mathbb{E}_a\mathbb{E}_b(-1)^{\tau(\nu(\t);a)+\tau(\nu(\s);b)}.
\end{equation}
Note that
\begin{equation}
\label{eq:brack_sq}
\text{TS} = \(\sum_{\nu(\t) \neq \bm{0}} \mathbb{E}_a (-1)^{\tau(\nu(\t);a)}\)^2.
\end{equation}
Exactly the same reasoning as was exploited to obtain (\ref{eq:dev_III_tag}) from (\ref{eq:III_1}) applies here and we get for the sum inside the brackets of (\ref{eq:brack_sq}),
\begin{equation}
\label{eq:sec_bound}
\sum_{\nu(\t) \neq \bm{0}} \mathbb{E}_a (-1)^{\tau(\nu(\t);a)} = O(n^{r/2}),
\end{equation}
and therefore,
\begin{equation}
\label{eq:ts}
\text{TS} = \(\sum_{\nu(\t) \neq \bm{0}} \mathbb{E}_a (-1)^{\tau(\nu(\t);a)}\)^2 = O(n^r).
\end{equation}

Overall, by plugging (\ref{eq:u_0}), (\ref{eq:u_g}) and (\ref{eq:ts}) into (\ref{eq:last_l_68}) we obtain
\begin{equation}
\text{S}_2 = O(n^r).
\end{equation}
Now impose back the conditions (\ref{eq:lin_cond}) on $\t$ and $\s$. Similarly to the proof of Proposition \ref{lem:main_exp}, this does not change the obtained bound exactly the same way as we explained above. Together with (\ref{eq:s_0}) we conclude from (\ref{eq:var_av_5}) that
\begin{equation}
\var{\beta_r(\A_n)} = \frac{O(n^r)}{2^{2r}n^{r+2}} = O\(\frac{1}{n^2}\).
\end{equation}
Now impose back conditions (\ref{eq:lin_cond}) on $\t$ and $\s$, which does not destroy the obtained bound exactly the same way as we explained in the proof of Proposition \ref{lem:main_exp}. This completes the proof.
\end{proof}

\begin{rem}
\label{rem:better_conv_var_pr}
As we have already mentioned in the proof of Proposition \ref{lem:main_exp}, the expectations over even cycles are not affected by the statistical dependencies between the matrix entries (edge weights) since all of them are raised into even powers. Therefore, the argument from Remark \ref{rem:better_conv_var} applies in the pseudo-random case as well, and we conclude that
\begin{equation}
\label{eq:better_conv_var_pr}
\var{\beta_r(\A_n)} = \Theta\(\frac{1}{n^2}\).
\end{equation}
\end{rem}

\bibliographystyle{IEEEtran}
\bibliography{ilya_bib}
\end{document}